\documentclass[10pt,preprint]{amsart}
\usepackage{amssymb,amsfonts,amsthm,amsmath,amscd,relsize}
\usepackage{hyperref}
\usepackage{enumerate}
\usepackage{comment}

 \usepackage{epsfig}
 \usepackage{graphicx}
 \usepackage{lipsum}
\usepackage{xcolor}
\usepackage{float}
\usepackage{soul}
\usepackage{adjustbox}

\usepackage{enumitem}

\usepackage{pgf,tikz,pgfplots}
\usepackage{xcolor}
\pgfplotsset{compat=1.13}
\usepackage{mathrsfs}
\usetikzlibrary{arrows}

\definecolor{qqqqff}{rgb}{0,0,1}
\definecolor{qqwuqq}{rgb}{0,0.39215686274509803,0}
\definecolor{uuuuuu}{rgb}{0.266,0.266,0.266}

\def\id{\operatorname{id}}
\def\N{\mathbb N}
\def\Z{\mathbb Z}

\def\a{\alpha}
\def\g{q}
\def\v{\varphi}
\def\k{\kappa}

\theoremstyle{plain}
\newtheorem{theorem}{Theorem}
\newtheorem*{Oz_1}{Probl\`eme XXIII}
\newtheorem*{Oz_2}{Probl\`eme XXVII}
\newtheorem{proposition}{Proposition}
\newtheorem{lemma}{Lemma}
\newtheorem{corollary}{Corollary}
\newtheorem*{RT}{Rosenberger's Theorem}

\theoremstyle{definition}
\newtheorem{definition}{Definition}

\theoremstyle{remark}
\newtheorem{remark}{Remark}
\newtheorem{example}{Example}
\newtheorem*{notation}{Notation}
\newtheorem*{thks}{Thanks}

\makeatletter
\@addtoreset{footnote}{page}
\makeatother

\let\svthefootnote\thefootnote

\begin{document}

\title{Lattice Equable Quadrilaterals I - Parallelograms}

\author{Christian Aebi and Grant Cairns$^*$}\let\thefootnote\relax\footnote{$*$ corresponding author}
\addtocounter{footnote}{-1}\let\thefootnote\svthefootnote

\address{Coll\`ege Calvin, Geneva, Switzerland 1211}
\email{christian.aebi@edu.ge.ch}
\address{Department of Mathematics, La Trobe University, Melbourne, Australia 3086}
\email{G.Cairns@latrobe.edu.au}

\begin{abstract}
This paper studies equable parallelograms whose vertices lie on the integer lattice. Using Rosenberger's Theorem on generalised Markov equations, we show that the g.c.d.~of the side lengths of such parallelograms can only be 3, 4 or 5, and in each of these cases the set of parallelograms naturally forms an infinite tree all of whose vertices  have degree 4, bar the root. The paper then focuses on what we call Pythagorean  equable parallelograms. These are lattice equable parallelograms whose complement in a circumscribing rectangle consists of two Pythagorean triangles. We prove that for these parallelograms the shortest side can only be  3, 4,  5, 6 or 10, and there are five infinite families of such parallelograms,  given by solutions to corresponding Pell-like equations. 
\end{abstract}

\maketitle

\section{Introduction} 
A polygon
 with integer sides is said to be \emph{equable} if its perimeter equals its area. Equable  polygons have  fascinated recreational mathematical amateurs at least as far back as the 17th century \cite{JO}:\

\begin{Oz_1}  D\'ecrire un triangle rectangle, dont l'aire soit en nombres \'egale au contour.\footnote{Describe a right triangle having the same area as perimeter.}
\end{Oz_1}

\begin{Oz_2} D\'ecrire  un parall\'elogramme rectangle, dont l'aire soit en nombres \'egale au contour. \footnote{Describe a rectangle having the same area as perimeter.}
\end{Oz_2}

All of the integer solutions given by Ozanam can easily be demonstrated in high school today, since they  rely merely on the ability of  \emph{completing a rectangle}, instead of the usual \emph{square}.

(a) For the Pythagorean triangle with side lengths $a<b<\sqrt{a^2+b^2}$ we have
 \begin{align*}
   \frac{1}{2}ab  =a+b+\sqrt{a^2+b^2} &\implies 0=(ab -2a-2b)^2-4(a^2+b^2)\\
   &\hskip1.15cm = ab(ab-4a-4b+8)\\
   & \implies (a-4)(b-4) =8, 
\end{align*}
and hence $(a,b)=(5,12)$ and $(6,8)$ are the only two possibilities.

(b)  For a rectangle with integer side lengths $a,b$  with $a\le b$ we have
  \[
 ab=2(a+b) \iff 0=ab -2a-2b = (a-2)(b-2) - 4 \iff (a-2)(b-2)=4,
  \]
giving $(a,b)=(4,4)$ and $(3,6)$.

The striking feature of equable  integer sided triangles is that apart from the two Pythagorean triangles given above, there are only 3 other possibilities; see the Appendix.
In this paper we study equable parallelograms. Unlike  equable  triangles,  at first sight equable parallelograms are sadly disappointing, as they are just far too common. Indeed, it is easier to itemise the non-equable cases.

\begin{proposition}\label{P:exist} If $a,b$ are positive integers with $a\le b$, then there is an equable parallelogram with sides $a,b$ unless  one of the following holds: 
\begin{enumerate}
\item[\rm(a\rm)]  $a=1$ or $2$, and $b$ arbitrary with $a\le b$.
\item[\rm(b\rm)] $a=3$ and $b=3,4$ or $ 5$.
\end{enumerate}
\end{proposition}

\begin{proof}
Repeating the argument we used above for Ozanam's Probl\`eme XXVII, the area is less than the perimeter when
\[
ab<2(a+b) \iff 4 > (a-2)(b-2),
\]
and this condition is satisfied for precisely those values of $a,b$ itemised in the possible conditions of the proposition. So in these cases, there is no equable parallelogram with side lengths $a,b$. Conversely, if  $ab\ge 2(a+b)$, start with the rectangle with side lengths $a,b$, and hence area $ab$, and gradually push the parallelogram over, while maintaining its side lengths, as in Figure \ref{F:shear}. Ultimately, when the parallelogram becomes flat,  the area is zero. So by continuity, somewhere in the process the area equals $2(a+b)$ and the parallelogram is equable.
\end{proof}

\begin{figure}[H]
\begin{tikzpicture}[scale=.6][line cap=round,line join=round,>=triangle 45,x=1cm,y=1cm]
%\clip(-0.27322468186995236,-0.9876247375534437) rectangle (6.392071733693616,4.607146058684385);
\draw[line width=0.8pt,color=qqwuqq,fill=qqwuqq,fill opacity=0.1] (2.73,0) -- (2.73,0.2703485445925461) -- (3,0.2703485445925461) -- (3,0) -- cycle; 
\draw [line width=0.8pt,color=qqqqff] (0,0)-- (0,4);
\draw [line width=0.8pt,color=qqqqff] (0,4)-- (3,4);
\draw [line width=0.8pt,color=qqqqff] (3,4)-- (3,0);
\draw [line width=0.8pt,color=qqqqff] (3,0)-- (0,0);
\draw [->,line width=0.8pt] (4.5,1) -- (5.5,1);
\draw [line width=0.8pt,color=qqqqff] (6,0)-- (6+2.64575,3);
\draw [line width=0.8pt,color=qqqqff] (6+2.64575,3)-- (9+2.64575,3);
\draw [line width=0.8pt,color=qqqqff] (9+2.64575,3)-- (9,0);
\draw [line width=0.8pt,color=qqqqff] (6,0)-- (9,0);
\draw [->,line width=0.8pt] (11.7,1) -- (12.7,1);
\draw [line width=0.8pt,color=qqqqff] (13,0)-- (13+3.4641,2);
\draw [line width=0.8pt,color=qqqqff] (13+3.4641,2)-- (16+3.4641,2);
\draw [line width=0.8pt,color=qqqqff] (16+3.4641,2)-- (16,0);
\draw [line width=0.8pt,color=qqqqff] (13,0)-- (16,0);
\draw (1.5,-.4) node {$a$};
\draw (7.5,-.4) node {$a$};
\draw (14.5,-.4) node {$a$};
\draw (-.4,2) node {$b$};
\draw (6.9,1.9) node {$b$};
\draw (14.6,1.5) node {$b$};
\end{tikzpicture}
\caption{Collapsing a  parallelogram}\label{F:shear}
\end{figure}
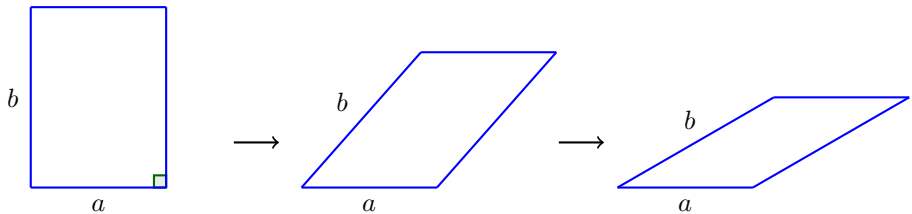

Given the above, to make life interesting we restrict ourselves to considering a subclass of equable parallelograms that has a wealth of interesting members, without becoming mundane.

\begin{definition}
A \emph{lattice equable parallelogram} (or \emph{LEP}, for short) is a parallelogram whose perimeter equals its area and whose vertices lie on the integer lattice $\Z^2$. 
\end{definition}

Notice that in the above definition, we do not require that the side lengths be integers. Indeed, as we show in Lemma \ref{L:bcint} below, this fact can be deduced from the other hypotheses. Throughout this paper we will denote the side lengths by the symbols $a$ and $b$. Moreover, for brevity, we will at times drop the word \emph{length} and simply refer to $a,b$ as the \emph{sides}. Note that a LEP is completely determined, up to a Euclidean motion, by its sides $a,b$. Indeed, if $\theta$ denotes one of the angles between the sides, then the area is $ab\sin\theta$ and so by equability, $\sin\theta=2(a+b)/ab$ is determined by $a$ and $b$. So our main aim is this paper in to study the values of $a,b$ for which a LEP exists with sides $a,b$. 

In Section \ref{S:props} we use a result of Paul Yiu on Heronian triangles to prove the following criteria. 

\begin{theorem}\label{T:suff}
Given positive integers $a,b$, a lattice equable parallelogram with sides $a,b$ exists if and only if  $a^2 b^2 -4(a+b)^2$ is a square.
\end{theorem}

In Section \ref{S:345} we use Rosenberger's Theorem on generalised Markov equations to prove  the following result.

\begin{theorem}\label{T:345}
 The set of ordered pairs $(a,b)$ of positive integers $a\le b$, for which $a^2 b^2 -4(a+b)^2$ is a square, is given by the disjoint union 
$\mathcal T_3 \cup \mathcal T_4 \cup \mathcal T_5$ where
\begin{align*}
\mathcal T_3 =\{(3q, 3r) \in\N^2 \ &|\  \exists (m,n) \in\N^2, m^2 + n^2 + q^2 = 3mnq
\ \&\   q \le r =3mn-q\},\\
\mathcal T_4 =\{(4q,4r)\in\N^2\ &|\ \exists (m,n) \in\N^2, m^2 + n^2+2q^2 =4mnq
\ \&\     q \le r =2mn-q\},\\
\mathcal T_5 =\{(5q,5r)\in\N^2\ &|\ \exists (m,n) \in\N^2, m^2 + n^2+5q^2 =5mnq
\ \&\     q \le r =mn-q\}.
\end{align*}
Furthermore, for each $i=3,4,5$, if $(a,b)\in \mathcal T_i$, then $\gcd(a,b)=i$.
\end{theorem}

Theorems \ref{T:suff} and \ref{T:345} provide a classification of LEPs,
but we need to clarify what we mean by \emph{classification}, as there is more than one natural notion. First, arguing as we explained above, it is clear that \emph{up to the full group of Euclidean motions}, a LEP is determined by its side lengths $a,b$. Alternately, one can consider the smaller group  $E(\Z^2)$ of isometries of the integer lattice $\Z^2$. This  is the semi-direct product of $\Z^2$ with its automorphism group, which is the dihedral group of order 8. Explicitly, $E(\Z^2)$ is generated by integral translations and by two generators of finite order, the rotation $(x, y) \mapsto (-x, y)$ and the reflection $(x, y) \mapsto (y, x)$. The set of LEPs are of course preserved by the action of the $E(\Z^2)$, so the orbits of LEPs define equivalence classes. It is natural to imagine  that there might  only be one orbit of LEPs with given side lengths $a,b$. In fact, this is not the case. For example,  the LEPs with vertices $(0,0),(6,8),(6,13),(0,5)$ and $(0,0),(10,0),(14,3),(4,3)$ both have side-lengths 5, 10,  but these LEPs do not belong to the same $E(\Z^2)$-equivalence class. So classification \emph{up to $E(\Z^2)$-equivalence} is more complicated than classification up to Euclidean motions. In this paper we will restrict ourselves to the latter, simpler notion of equivalence, so that for us a LEP will be determined by its side lengths. For this reason, from Section~\ref{S:forest} onwards, we will identify a LEP with its corresponding pair $(a,b)$ of integers satisfying the condition of Theorem~\ref{T:suff}.

In Section \ref{S:forest} we describe how all LEPs can be derived from three fundamental examples by successive applications of four functions. This gives the forest of LEPs, consisting of three trees corresponding to the three possible values of $\gcd(a,b)$: see Figures \ref{F:tree35} and \ref{F:tree4}. 
Figure \ref{F:base} shows 9 LEPs, three from each of the trees.
The top-left most three are the root with side lengths $(3,6)$ and the adjacent LEPs  with side lengths $(3,15)$ and $(6,39)$.
The middle three are the root with side lengths $(4,4)$, the adjacent LEP  with side lengths $(4,20)$, and the LEQ adjacent to that with side lengths $(20,116)$.
The bottom three are the root with side lengths $(5,5)$ and the adjacent LEPs  with side lengths $(5,10)$ and $(5,85)$.

\begin{figure}[h]
\begin{tikzpicture}[scale=.1,line cap=round,line join=round,>=triangle 45,x=1cm,y=1cm]
\begin{axis}[
x=1cm,y=1cm,
axis lines=middle,
ymajorgrids=true,
xmajorgrids=true,
xmin=0,
xmax=100,
ymin=-10,
ymax=100,
xtick={0,10,...,100},
ytick={-10,0,...,100},]
\def\xg{10};
\def\yg{-10};
\def\xr{0};
\def\yr{30};
%\clip(-68.94007264785111,-45.56473111711375) rectangle (122.02265807497092,109.1482504734607);
\fill[shift={(\xg,\yg)},line width=2pt,color=green] (0,0) --(3,-4) -- (8,-4) -- (5,0) -- cycle;
\fill[shift={(\xg,\yg)},line width=2pt,color=green!40] (8,-4) -- (18,-4) -- (15,0) -- (5,0) -- cycle;
\fill[shift={(\xg,\yg)},line width=2pt,color=green!15] (0,0) -- (5,0) -- (82,36) -- (77,36) -- cycle;
\fill[shift={(\xr,\yr)},line width=2pt,color=red!20] (3,-6) -- (18,30) -- (18,36) -- (3,0) -- cycle;
\fill[shift={(\xr,\yr)},line width=2pt,color=red!200] (0,0) -- (3,0) -- (3,-6) -- (0,-6) -- cycle;
\fill[shift={(\xr,\yr)},line width=2pt,color=red!45] (0,-6) -- (9,-18) -- (12,-18) -- (3,-6) -- cycle;
\fill[line width=2pt,color=blue!150] (4,-4) -- (0,-4) -- (0,0) -- (4,0) -- cycle;
\fill[line width=2pt,color=blue!20] (0,0) -- (16,12) -- (96,96) -- (80,84) -- cycle;
\fill[line width=2pt,color=blue!45] (0,0) -- (4,0) -- (20,12) -- (16,12) -- cycle;
\draw[shift={(\xg,\yg)},line width=2pt] (0,0) --(3,-4) -- (8,-4) -- (5,0) -- cycle;
\draw[shift={(\xg,\yg)},line width=2pt] (8,-4) -- (18,-4) -- (15,0) -- (5,0)  -- cycle;
\draw[shift={(\xg,\yg)},line width=2pt] (0,0) -- (5,0) --(82,36) -- (77,36)-- cycle;
\draw[shift={(\xr,\yr)},line width=2pt] (3,-6) -- (18,30) -- (18,36) -- (3,0)  -- cycle;
\draw[shift={(\xr,\yr)},line width=2pt] (0,0) -- (3,0) -- (3,-6) -- (0,-6)  -- cycle;
\draw[shift={(\xr,\yr)},line width=2pt] (0,-6) -- (9,-18) -- (12,-18) -- (3,-6) -- cycle;
\draw[line width=2pt] (4,-4) -- (0,-4) -- (0,0) -- (4,0) -- cycle;
\draw[line width=2pt] (0,0) -- (16,12) -- (96,96) -- (80,84) -- cycle;
\draw[line width=2pt] (0,0) -- (4,0) -- (20,12) -- (16,12) -- cycle;
\end{axis}
\end{tikzpicture}
\caption{Nine LEPs, three from each tree}\label{F:base}
\end{figure}
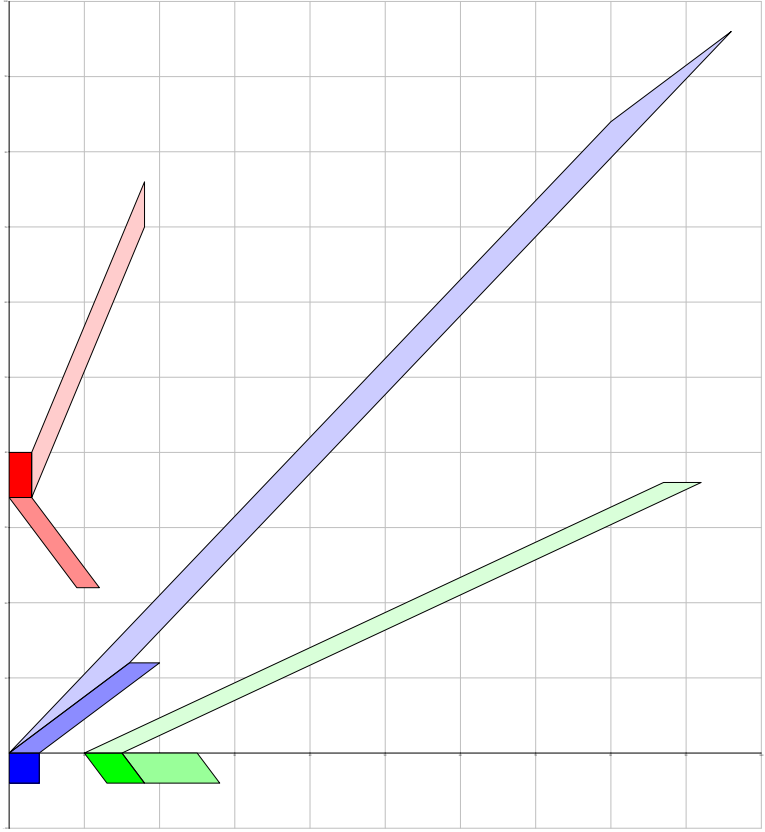

We investigate various aspects and applications of these results in Sections \ref{S:rec} and \ref{S:dah}. 
Then in Section \ref{S:PEPs} we analyse a large natural  family of LEPs that we call \emph{Pythagorean  equable parallelograms}. These are LEPs whose complement in a circumscribing rectangle consists of two Pythagorean triangles; see Definition \ref{D:PEPs}. We use Theorem \ref{T:345} to prove the following: 

\begin{theorem}\label{T:PEPs}
If a Pythagorean equable parallelogram has sides $a,b$ with $a\le b$ then $a=3,4,5,6$ or $10$. There are 5 infinite families of such parallelograms, which are given by the solutions of corresponding Pell or Pell-like equations: 
\begin{enumerate}
\item[{\emph{(F1)}}]  $a=3$, $b=\frac{3(x^2+y^2)}{2}$, where  $y^2-5x^2 =4$.
\item[{\emph{(F2)}}] $a=4$, $b=4(x^2+y^2)$, where  $y^2-3x^2 =1$.
\item[{\emph{(F3)}}] $a=5$, $b=\frac{x^2+y^2}{2}$, where  $3y^2-7x^2 =20$.
\item[{\emph{(F4)}}]  $a=6$, $b=3(x^2+y^2)$, where  $y^2-2x^2 =1$.
\item[{\emph{(F5)}}] $a=10$, $b=x^2+y^2$, where  $2y^2-3x^2 =5$.
\end{enumerate}
\end{theorem}

At the end of Section \ref{S:PEPs} we locate the Pythagorean equable parallelograms on certain branches of the trees of LEPs given in Figures \ref{F:tree35} and \ref{F:tree4}. 

The paper concludes with an appendix that revisits the classic theorem  on equable triangles that dates to 1904.

\begin{notation} In this paper, we employ the term \emph{positive} in the strict sense. So $\N=\{n\in\Z \ |\ n>0\}$.
\end{notation}

%%%%%%%%%%%%%%%%%%%%

\section{LEPs:  general properties and special cases}\label{S:props}

Let us begin with a trivial remark that we will use on several occasions.

\begin{remark}\label{L:int}
If the square root of an integer is rational, then it is an integer. 
Consequently, if the distance $d$ between two integer lattice points is rational, then $d$ is an integer.
\end{remark}

\begin{lemma}\label{L:bcint}
LEPs have integer side lengths.\end{lemma}

\begin{proof}
Consider a LEP $P$ with vertices 
$O(0,0), A(x,y),B(z,w),C(u,v)$,
in anticlockwise order, where $z=x+u,w=y+v$. Let $a$ denote the length of $OA$ and $b$ the length of $OC$.
The area of $P$ is $xv-yu$, which is an integer. By the equability hypothesis, $2(a+b)$ is an integer. So, as $a^2,b^2$ are integers,
$a-b=(a^2-b^2)/(a+b)$ is rational. Thus $a=\frac{a+b}2+\frac{a-b}2$ is rational, and hence $a$ is an integer, by Remark~\ref{L:int}. 
So $b$ is also rational, and hence an integer.
\end{proof}

\begin{remark}
In fact, the above lemma is a special case of a general result: every lattice equable polygon has integer sides. Indeed, lattice polygons can be partitioned into lattice triangles, and every lattice triangle has integer or half-integer area, by the same argument we used in the proof of the above lemma. So by the equability hypothesis, the sum of the side lengths of a lattice equable polygon is rational. But each side length is a square root of an integer, and it is well known that if $\sum_{i=1}^n\sqrt{a_i}$ is rational for integers $a_1,\dots,a_n$, then $a_i$ is rational for each $i$; see for example \cite{art} or \cite{Yuan}. But then the side lengths are all integers by Remark~\ref{L:int}.
\end{remark}

\begin{proposition}\label{P:partit}
No LEP can be partitioned into two congruent right triangles.
\end{proposition}

\begin{proof}
Suppose a LEP $P$ can be partitioned into two congruent right triangles. Then $P$ can be cut in half and recombined to form an equable isosceles triangle with integer sides, as in Figure \ref{F:PneP}. The proposition then follows from the fact that the complete list of 5 equable triangles includes no isosceles triangles; see the Appendix below. 

\begin{figure}[H]
\begin{tikzpicture}[scale=.6][line cap=round,line join=round,>=triangle 45,x=1cm,y=1cm]
%\clip(-0.27322468186995236,-0.9876247375534437) rectangle (6.392071733693616,4.607146058684385);
\draw[line width=0.8pt,color=qqwuqq,fill=qqwuqq,fill opacity=0.1] (2.73,0) -- (2.73,0.2703485445925461) -- (3,0.2703485445925461) -- (3,0) -- cycle; 
\draw[line width=0.8pt,color=qqwuqq,fill=qqwuqq,fill opacity=0.1] (3.270348544592546,4) -- (3.270348544592546,3.73) -- (3,3.73) -- (3,4) -- cycle; 
\draw [line width=0.8pt,color=qqqqff] (0,0)-- (3,4);
\draw [line width=0.8pt,color=qqqqff] (3,4)-- (6,4);
\draw [line width=0.8pt,color=qqqqff] (3,0)-- (0,0);
\draw [line width=0.8pt,color=qqqqff] (3,0)-- (6,4);
\draw [line width=0.8pt,dash pattern=on 1pt off 1pt] (3,4)-- (3,0);
\draw (1.5,-.4) node {$a$};
\draw (2.7,2) node {$h$};
\draw (1.5,2.5) node {$b$};
\draw [->,line width=0.8pt] (6.5,2) -- (7.5,2);
\draw[line width=0.8pt,color=qqwuqq,fill=qqwuqq,fill opacity=0.1] (11.27,0) -- (11.27,0.27) -- (11,0.27) -- (11,0) -- cycle; 
\draw [line width=0.8pt,color=qqqqff] (8,0)-- (11,4);
\draw [line width=0.8pt,color=qqqqff] (11,4)-- (14,0);
\draw [line width=0.8pt,color=qqqqff] (14,0)-- (8,0);
\draw [line width=0.8pt,dash pattern=on 1pt off 1pt] (11,4)-- (11,0);
\draw (9.5,-.4) node {$a$};
\draw (10.7,2) node {$h$};
\draw (9.5,2.5) node {$b$};
\end{tikzpicture}
\caption{A particular non-equable parallelogram}\label{F:PneP}
\end{figure}
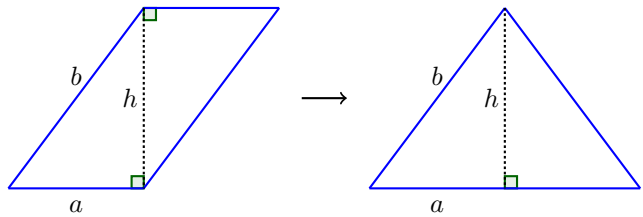

A simple and more direct proof of the proposition is as follows. The equability hypothesis gives
\begin{align}
ha=2a+2\sqrt{a^2+h^2}  \implies a^2(h-2)^2=4(a^2+h^2) &\iff 0=a^2h-4a^2-4h  \nonumber\\
&\iff  (a^2-4)(h-4)=16.\label{E:bh}
\end{align}
Consequently $h$ is rational and so by Remark~\ref{L:int}, $h$ is an integer. Then by \eqref{E:bh}, $a^2-4$ is a factor of 16. But there is obviously no such positive integer $a$. \end{proof}

\begin{lemma}\label{L:diag}
Suppose a LEP $P$ has  sides $a,b$. Then the lengths of the diagonals of $P$ are given by the following formula:
\[
d^2=(a^2  +  b^2)\pm 2\sqrt{a^2 b^2 -4(a+b)^2}.
\]
In particular,  $a^2 b^2 -4(a+b)^2$ is a square.
\end{lemma}

\begin{proof}
Consider a diagonal of length $d$. By Heron's formula, the triangle with sides $a,b,d$ has area
\[
\frac14\sqrt{(a+b+d)(-a+b+d)(a-b+d)(a+b-d)}.
\]
Hence the equability hypothesis is $(a+b+d)(-a+b+d)(a-b+d)(a+b-d)=16(a+b)^2$. Rearranging, this gives $d^4 - 2 (a^2  +  b^2)d^2 + (a^2  -  b^2)^2 +16(a+b)^2=0$,
so $d^2=(a^2  +  b^2)\pm \sqrt{(a^2  +  b^2)^2-(a^2  -  b^2)^2 -16(a+b)^2}$. Simplifying, we have
\[
d^2=(a^2  +  b^2)\pm \sqrt{4a^2 b^2 -16(a+b)^2},
\]
as required. In particular, as $a,b,d^2$ are integers, $4a^2 b^2 -16(a+b)^2$ is a square.
\end{proof}

\begin{comment}
\begin{align*}
d^2&=(a^2  +  b^2)\pm +{4a^2 b^2 -16(a+b)^2}\\
&=k^2m^2n^2-2k(m^2+n^2)+2k(n^2-m^2)\\
&=k^2m^2n^2-4km^2\\
&=m^2(k^2n^2-4k).
\end{align*}

\begin{align*}
d^2&=(a^2  +  b^2)- \sqrt{4a^2 b^2 -16(a+b)^2}\\
&=k^2m^2n^2-2k(m^2+n^2)-2k(n^2-m^2)\\
&=k^2m^2n^2-4kn^2\\
&=n^2(k^2m^2-4k).
\end{align*}
\end{comment}

\begin{remark}
Suppose a LEP $P$ has  sides $a,b$ and diagonals $d_1,d_2$. Then the above lemma gives
$d_1d_2=(a  +  b)\sqrt{16+(a-b)^2}$.
\end{remark}

\begin{comment}
Checking:
\begin{align*}
(d_1d_2)^2&=m^2(k^2n^2-4k)n^2(k^2m^2-4k)\\
&=(a  +  b)^2(kn^2-4)(km^2-4)\\
&=(a  +  b)^2(kn^2-4)(km^2-4)\\
&=(a  +  b)^2(k^2m^2n^2-4k(m^2+n^2)+16)\\
&=(a  +  b)^2(16+(a-b)^2).
\end{align*}
\end{comment}

\begin{proof}[Proof of Theorem \ref{T:suff}]
The necessity of the condition was shown in Lemma \ref{L:diag}. Therefore, assume that $a,b$ are positive integers such that $4a^2 b^2 -16(a+b)^2$ is a square.
Consider a triangle $T$ with sides $a,b$ and $d:=\sqrt{a^2  +  b^2+ 2\sqrt{a^2 b^2 -4(a+b)^2}}$. Notice that such a triangle exists because $a,b < d <a+b$, where the latter inequality holds as 
$a^2  +  b^2+ 2\sqrt{a^2 b^2 -4(a+b)^2} <(a+b)^2$ since $\sqrt{a^2 b^2 -4(a+b)^2} <ab$.
Let $\theta$ denote the angle between sides $a,b$ and note that $\theta$ is obtuse since $d^2\ge a^2+b^2$. So 
\[
d^2=a^2+b^2-2ab\cos\theta=a^2+b^2+2\sqrt{a^2b^2-a^2b^2\sin^2\theta}.
\]
So, from the definition of $d$, we have $ab\sin\theta= 2(a+b)$. Hence, since $ab\sin\theta$ is twice the area of $T$, the area of $T$ is $a+b$. Now consider the parallelogram $P$ made from two copies of $T$. From what we have just seen, $P$ is equable. It remains to show that $P$ can be realised as a lattice parallelogram, or equivalently, that  $T$ can be realised as a lattice triangle. But since $a,b$ are integers, and $d^2$ is an integer, and the area of $T$ is an integer, $T$ is \emph{geodetic}, in the terminology of Paul Yiu. (A triangle is geodetic if  
it has rational area and side lengths $\sqrt{X}, \sqrt{Y}, \sqrt{Z}$ for integers $X, Y, Z$). Thus $T$ can be realised as a lattice triangle; see the last paragraph of \cite{Yiu}.
\end{proof}

\begin{remark}
Observe that LEPs are not determined up to Euclidean motion by their area. Indeed, by Theorem \ref{T:suff}, there are LEPs with $(a,b)=(3,87), (5,85),(25,65)$ that each have area 180. 
A larger example is given by $(a,b)=(85,1525)$ and $(445,1165)$, which each have area 1610. 

\end{remark}

\begin{corollary}\label{C:rhom}\ 
\begin{enumerate}
\item[{\emph{(a)}}]  The only rhombi that are LEPs are the $4\times4$ square and the rhombus with side length 5 and area 20.
\item[{\emph{(b)}}]  If a LEP  has  sides $a,b$ with  $b=2a$, then $a=3$ or $5$.
\end{enumerate}
\end{corollary}

\begin{proof} From Theorem \ref{T:suff}, a rhombus with side length $a$ is a LEP if and only if  $a^4  -16a^2$ is a square, that is, if $a^2  -16$ is a square. But obviously this only occurs when $a$ is 4 or 5. The first case is the $4\times4$ square, while the second case is exhibited in Figure \ref{F:rhom}. Similarly,  if $b=2a$, then $4a^4-36a^2$ is a square and so $a^2-9$ is a square. Hence $a=3$ or $5$.
\end{proof}

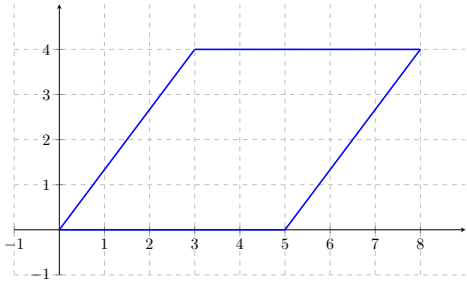
\begin{figure}[h]
\begin{tikzpicture}[scale=.6][line cap=round,line join=round,>=triangle 45,x=1cm,y=1cm]
\begin{axis}[
x=1cm,y=1cm,
axis lines=middle,
grid style=dashed,
ymajorgrids=true,
xmajorgrids=true,
xmin=-1,
xmax=9,
ymin=-1,
ymax=5,
xtick={-1,0,...,8},
ytick={-1,0,...,4},]
\clip(-3.5730154925130506,-1.5) rectangle (11.024141051110055,8.447022588171118);
\draw [line width=1pt,color=qqqqff] (8,4)-- (3,4);
\draw [line width=1pt,color=qqqqff] (3,4)-- (0,0);
\draw [line width=1pt,color=qqqqff] (0,0)-- (5,0);
\draw [line width=1pt,color=qqqqff] (5,0)-- (8,4);
\end{axis}
\end{tikzpicture}
\caption{Equable Rhombus}\label{F:rhom}
\end{figure}

We conclude this section with an elementary result whose proof will hopefully give the reader a better feel for the nature of LEPs.

\begin{theorem}\label{T:symm}
There are only three LEPs lying in the first quadrant with one vertex at the origin and a diagonal lying on the line $x=y$. They are the $4\times 4$ square, the LEP with vertices  $(0,0),(3,0),(12,12),(9,12)$, and its reflection in the line $y=x$.
\end{theorem}

\begin{proof}
Consider a parallelogram $P$ with vertices 
$O(0,0), A(x,y),B(z,z),A'(z-x,z-y)$,
in clock-wise order, where $x,y,z$ are non-negative integers and $y<x\le z$, as in diagram on the left of Figure \ref{F:symm}. Let us first calculate the area $\alpha$ of the triangle $OAB$. The diagonal $OB$ has length $\sqrt{2}z$ and the distance of $A$ to the diagonal is $(x-y)/\sqrt2$. So $\alpha=\frac12z(x-y)$. The sum $\sigma$ of the lengths $OA$ and $AB$ is
\begin{equation}\label{E:s}
\sigma=\sqrt{x^2+y^2}+\sqrt{(z-x)^2+(z-y)^2}.\end{equation}
Moreover, $\sigma$ is no greater than the sum of the lengths of the segments $0(z,0)$ and $(z,0)B$. That is, $\sigma\le 2z$. Assume that $P$ is a LEP. The equable hypothesis, $\alpha=\sigma$,  gives 
\begin{equation}\label{E:eqq}
\frac12z(x-y)\le 2z,
\end{equation}
and hence $y<x\le y+4$. So we have 4 cases to consider:

\begin{figure}[h]
\begin{tikzpicture}[scale=.75,line cap=round,line join=round,>=triangle 45,x=1cm,y=1cm]
\pgfplotsset{ticks=none}
\begin{axis}[
x=1cm,y=1cm,
axis lines=middle,
grid style=dashed,
ymajorgrids=true,
xmajorgrids=true,
xmin=-1.1,
xmax=6.1357147342069185,
ymin=-1,
ymax=6.5,
xtick={0,1,...,6},
ytick={0,1,...,6},]
%\clip(-1.8,-0.8533908838021603) rectangle (6.1357147342069185,5.971521324571661);
\draw [line width=0.8pt,dotted,domain=-1.267579864707057:6.1357147342069185] plot(\x,{(-0-1*\x)/-1});
\draw [line width=1pt,color=qqqqff] (3,1)-- (5,5);
\draw [line width=1pt,color=qqqqff] (0,0)-- (3,1);
\draw [line width=1pt,color=qqqqff] (5,5)-- (2,4);
\draw [line width=1pt,color=qqqqff] (2,4)-- (0,0);
\draw (5,-0.26) node {$(z,0)$};
%\begin{scriptsize}
\draw [fill=uuuuuu] (0,0) circle (1.5pt);
\draw [color=black] (-0.5542415830408668,0.17805771266111486) node {$O(0,0)$};
\draw [fill=black] (3,1) circle (1.5pt);
\draw[color=black] (3.6872292809202647,0.8913959943273053) node {$A(x,y)$};
\draw [fill=black] (5,5) circle (1.5pt);
\draw[color=black] (4.6319205188025165,5.248543336396468) node {$B(z,z)$};
\draw [fill=uuuuuu] (2,4) circle (1.5pt);
\draw[color=black] (1.7014497400657351,4.4195285766222465) node {$A'(z-x,z-y)$};
%\end{scriptsize}
\end{axis}
\end{tikzpicture}
\hskip .4cm
\begin{tikzpicture}[scale=.375][line cap=round,line join=round,>=triangle 45,x=1cm,y=1cm]
\begin{axis}[
x=1cm,y=1cm,
axis lines=middle,
grid style=dashed,
ymajorgrids=true,
xmajorgrids=true,
xmin=-2,
xmax=13,
ymin=-2,
ymax=13,
xtick={0,2,...,12},
ytick={0,2,...,12},]
%\clip(-2,-2) rectangle (14,14);
\draw [line width=2pt,color=qqqqff] (12,12)-- (9,12);
\draw [line width=2pt,color=qqqqff] (9,12)-- (0,0);
\draw [line width=2pt,color=qqqqff] (0,0)-- (3,0);
\draw [line width=2pt,color=qqqqff] (3,0)-- (12,12);
\draw [dashed,line width=.5pt] (0,0)-- (12,12);
\end{axis}
\end{tikzpicture}
\caption{Parallelograms with diagonal on line $y=x$}\label{F:symm}
\end{figure}
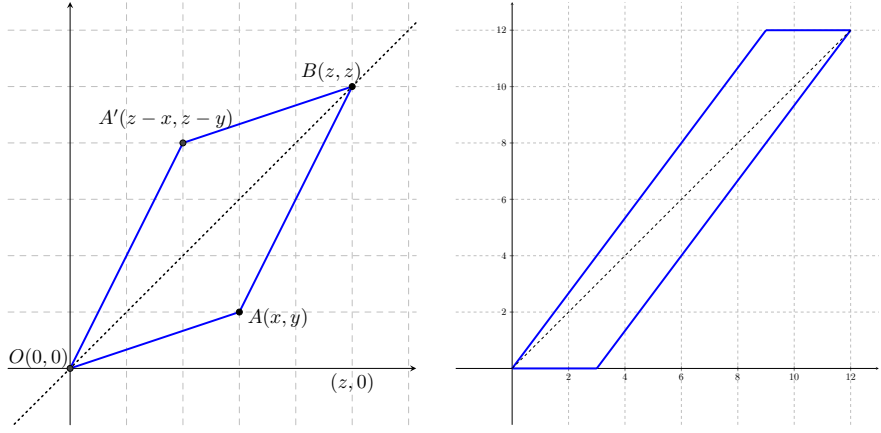

 Case $x= y+4$. Here \eqref{E:eqq} is an equality, from which we conclude that the given parallelogram is a square, with side length $z$. So, by equability, $z^2=4z$ and hence $z=4$. So $A=(4,0),B=(4,4)$.

Case $x= y+3$. Substituting this in \eqref{E:s}, equability gives
\[
\frac32z=\sqrt{x^2+(x-3)^2}+ \sqrt{(z-x)^2+(z-x+3)^2}.
\]
Writing this as $3z-2\sqrt{x^2+(x-3)^2}= 2\sqrt{(z-x)^2+(z-x+3)^2}$ and squaring both sides  and simplifying gives
$9z^2  -12z\sqrt{x^2+(x-3)^2} =4(2z^2   -2z(2x-3) )$.  
Rearranging and dividing by $z$ gives
$z = 4(6 - 4 x + 3 \sqrt{x^2+(x-3)^2})$.
In particular, $x^2+(x-3)^2$ is a  square. Moreover, $x\le z$ gives 
$4(6 - 4 x + 3 \sqrt{x^2+(x-3)^2})\ge x$
 and so
$144(x^2+(x-3)^2)\ge (17x-24)^2$
and simplifying, 
\[
0\ge x^2 +48x -720=(x+60)(x-12).
\]
Thus $x\le12$. Calculations show that the only such values for which $x^2+(x-3)^2$ is a  square are $x=3$ and $x=12$. These values correspond to the parallelogram $(0,0),(3,0),(12,12),(9,12)$, and its reflection in the line $y=x$; see the diagram on the right of Figure \ref{F:symm}.

Cases $x= y+2$ and $x= y+1$.  The area is respectively $z$ and $z/2$ in these cases, but the sum $\sigma$ of the sides is greater than the length of the diagonal, so $\sigma>\sqrt2 z$. Thus equability is impossible in both cases.
\end{proof}

\begin{remark}
The above result does not hold without the assumption that the LEPs lies in the first quadrant. For example, consider the following 6 LEPs with vertices $O(0,0), A(x,y),B(z,z),A'(z-x,z-y)$,
in clock-wise order:
\begin{align*}
&A=(0,-3),\ B=(60,60),\\
&A=(0,-4),\ B=(12,12),\\
&A=(-12,-16),\ B=(68,68),\\
&A=(-9,-12),\ B=(348,348),\\
&A=(-60,-63),\ B=(2028,2028),\\
&A=(-80,-84),\ B=(396,396).
\end{align*}
Notice incidentally that in all these cases $x-y\le 4$, as in the proof of the above theorem, even though in these cases the diagonal  on $y=x$ is the short diagonal of the LEP. This is a general fact; see Theorem \ref{T:altbd} below.
\end{remark}

%%%%%%%%%%%%%%%%%%%%%%%%%

\section{The LEP restriction: $\gcd(a,b)=3,4,5$}\label{S:345}

The aim of this section is to prove Theorem \ref{T:345}. We first show that the elements of the sets $\mathcal T_3,\mathcal T_4,\mathcal T_5$ have the required property.
 First suppose that $(a,b)=(3q, 3r)\in \mathcal T_3$, where for positive integers $m,n$ we have
 $ m^2 + n^2 + q^2 = 3mnq$ and $q <r =3mn-q$. 
  Hence $q,r$ are the two solutions of the quadratic equation
 \[
 Z^2-3mnZ+ (m^2+n^2)=0,
 \]
 in $Z$, and so $q+r=3mn$ and $qr=m^2+n^2$. Thus we have
 \[
 a^2b^2-4(a+b)^2=9^2q^2r^2-4\cdot 9(q+r)^2=9^2((m^2+n^2)^2-4m^2n^2)=9^2(m^2-n^2)^2,
 \]
 which is a square, as required. Similarly, suppose that $(a,b)=(4q, 4r)\in \mathcal T_4$, where for positive integers $m,n$ we have
 $ m^2 + n^2 + 2q^2 = 4mnq$ and $q <r =2mn-q$. Here $q,r$ are the two solutions of the quadratic equation
 \[
 2Z^2-4mnZ+ (m^2+n^2)=0,
 \]
 and so $q+r=2mn$ and $qr=\frac12(m^2+n^2)$. Thus we have
 $a^2b^2-4(a+b)^2=8^2(m^2-n^2)^2$,
 which is again a square. Finally, suppose that $(a,b)=(5q, 5r)\in \mathcal T_5$, where for positive integers $m,n$ we have
 $ m^2 + n^2 + 5q^2 = 5mnq$ and $q <r =mn-q$. Here $q,r$ are the two solutions of the quadratic equation
 \[
 5Z^2-5mnZ+ (m^2+n^2)=0,
 \]
 and so $q+r=mn$ and $qr=\frac15(m^2+n^2)$. Thus we have
 $a^2b^2-4(a+b)^2=5^2(m^2-n^2)^2$,
 which is again a square.

We now show that every solution is in $\mathcal T_3\cup \mathcal T_4\cup \mathcal T_5$.
The main tool we use  is Rosenberger's Theorem on generalised Markov equations. Recall that in \cite{Ro} Rosenberger considered equations of the form 
\begin{equation}\label{E:Ro}
ax^2+by^2+cz^2=dxyz,
\end{equation}
where $a,b,c$ are pairwise relatively prime positive integers with $a\le b\le c$ such that $a,b,c$ all divide $d$. We are only interested in positive integer solutions, that is, $x,y,z\in\N$, so we use the word \emph{solution} to mean positive integer solution. Rosenberger's remarkable result is that only 6 such equations have a solution and when such a solution exists, there are infinitely many solutions. We use the R1--R5 notation of \cite{BU}.

\begin{RT}[\cite{Ro}]
Equation \eqref{E:Ro} only has a solution in the following 6 cases:
\begin{enumerate}
\item[{\emph{M:}}] $x^2 + y^2 + z^2 = 3xyz$ (Markov's equation),
\item[{\emph{R1:}}] $x^2 + y^2 + 2z^2 = 4xyz$,
\item[{\emph{R2:}}]  $x^2 + 2y^2 + 3z^2 = 6xyz$,
\item[{\emph{R3:}}]  $x^2 + y^2 + 5z^2 = 5xyz$,
\item[{\emph{R4:}}]  $x^2 + y^2 + z^2 = xyz$,
\item[{\emph{R5:}}]  $x^2 + y^2 + 2z^2 = 2xyz$,
\end{enumerate}
\end{RT}

\begin{remark}\label{R:equivs}
As mentioned in \cite{Ro}, it is easy to see by considering R4 modulo 3, that if $(x,y,z)$ is a solution to R4, then $x,y,z$ are each divisible by 3, and  $(x,y,z)$ is a solution to  Markov's equation if and only if $(3x,3y,3z)$ is a solution to R4.  
Similarly, by considering R5 modulo 4,  if $(x,y,z)$ is a solution to R5, then $x,y,z$ are each even, and $(x,y,z)$ is a solution to R1 if and only if $(2x, 2y, 2z)$ is a solution to R5. 
Concerning the other cases, it is easy to see from the theorem that for M, R1, R2 and R3, if $(x,y,z)$ is a solution, then $x,y,z$ are relatively prime.
\end{remark}

Returning to LEPs, suppose that $a,b$ are positive integers such that $a^2 b^2 -4(a+b)^2$ is a square, and $a\le b$. Using the standard characterisation of Pythagorean triples  we have relatively prime positive integers $m,n$ and a  positive integer $k$ such that either
\begin{equation}\label{E:py1}
ab= k(m^2+n^2), \quad 2(a+b)=2kmn,\quad \sqrt{a^2 b^2 -4(a+b)^2}= k(n^2-m^2),
\end{equation}
or 
\begin{equation}\label{E:py2}
ab= k(m^2+n^2), \quad \sqrt{a^2 b^2 -4(a+b)^2}=2kmn,\quad 2(a+b)= k(n^2-m^2).
\end{equation}
We will show that the second possibility can be reduced to the first. Indeed, suppose that \eqref{E:py2} holds. Assume for the moment that $k$ is odd. As $\sqrt{a^2 b^2 -4(a+b)^2}=2kmn$,  it follows that  $ab$ is even.
Consider  the equation $ab= k(m^2+n^2)$. As $ab$ is even and $k$ is odd and $m,n$ are relatively prime, $m,n$ must both be odd. Thus, from $ab= k(m^2+n^2)$, we have $ab\equiv 2 \pmod 4$. So exactly one of $a,b$ is even. Thus $2(a+b) \equiv 2 \pmod 4$. But as $m,n$ are all odd, $n^2-m^2\equiv 0 \pmod 4$, contradicting $2(a+b)= k(n^2-m^2)$.
Thus $k$ is even, $k=2\k$ say. Let $y=m+n,x=n-m$, so $n=\frac{y+x}2,m=\frac{y-x}2$. Then \eqref{E:py2} can be written
\[
ab=k(m^2+n^2)=\k(x^2+y^2),\quad  2(a+b)=k(n^2-m^2)=2\k xy,\quad 2kmn=\k(y^2-x^2),
\]
which has the same form as \eqref{E:py1}. Since $m,n$ are relatively prime, $\gcd(x,y)=1$ or $2$. In the latter case we can replace $x$ and $y$ by $x'=x/2$ and $y'=y/2$ respectively, and replace $\kappa$ by $\kappa'=4\kappa$. We then have the same form as \eqref{E:py1} again but now $x',y'$ are relatively prime. We may therefore assume \eqref{E:py1} for what follows.

 Note that $a,b$ are solutions to the equation
\begin{equation}\label{E:xbc}
x^2-kmn x+k(m^2+n^2)=0.
\end{equation}
In particular, we have
\begin{equation}\label{E:xb}
km^2+kn^2+a^2=kamn.
\end{equation}
Let $k=fs^2$, where $f$ is square-free.
%\begin{lemma}\label{L:589} $k$ is either 5, 8 or 9.
From \eqref{E:xb}, $f$ divides $a^2$ and hence $f$ divides $a$. Let $a=f\a$. Dividing \eqref{E:xb} by $f$ gives
\begin{equation}\label{E:xb2}
s^2m^2+s^2n^2+f\a^2=fs^2 mn\a.
\end{equation}
From \eqref{E:xb2}, $s^2$ divides $f\a^2$ and hence $s^2$ divides $\a^2$, and thus $s$ divides $\a$. Let $\a=s\g$.  Dividing \eqref{E:xb2} by $s^2$ gives
\begin{equation}\label{E:xb3}
m^2+n^2+f\g^2=f\!s\, mn\g.
\end{equation}
Thus by Rosenberger's Theorem, using $x=m,y=n,z=\g$, there are five possibilities: 
\begin{itemize}
%[leftmargin=*]
\item \eqref{E:xb3} is M, $f=1, s=3,k=9, a=3q$ and $b=kmn-a=3(3mn-q)$.
\item \eqref{E:xb3} is R1, $f=2, s=2,k=8, a=4q$ and $b=kmn-a=4(2mn-q)$.
\item  \eqref{E:xb3} is R3, $f=5, s=1,k=5, a=5q$ and $b=kmn-a=5(mn-q)$.
\item \eqref{E:xb3} is R4, $f=1, s=1,k=1, a=q$ and $b=kmn-a=mn-q$.
\item \eqref{E:xb3} is R5, $f=2, s=1,k=2, a=2q$ and $b=kmn-a=2(mn-q)$. 
\end{itemize}
In the first three cases we have $(a,b)\in \mathcal T_3,\mathcal T_4,\mathcal T_5$ respectively as in Theorem \ref{T:345}. We now eliminate the last two cases.
In the second last case,  \eqref{E:xb3} is $m^2+n^2+\g^2= mn\g$. But then as noted in Remark \ref{R:equivs}, the values $m,n,\g$ are all divisible by 3, which contradicts our assumption that $m,n$ are relatively prime.
Similarly, in the last case,  \eqref{E:xb3} is $m^2+n^2+2\g^2=2 mn\g$. But then as noted in Remark \ref{R:equivs}, the values  $m,n,\g$ are all even, again contradicting our assumption that $m,n$ are relatively prime.

 It remains to show that for every solution $(a,b)$, one has $\gcd(a,b)=3,4$ or 5 (thus showing that $\mathcal T_3,\mathcal T_4,\mathcal T_5$ are disjoint).
 As before, we have relatively prime $m,n$ with
\[
ab= k(m^2+n^2),\quad 2(a+b)=2kmn,\quad \sqrt{a^2 b^2 -4(a+b)^2}= k(n^2-m^2),
\]
where from above, $k$ is either 5, 8 or 9.

\begin{lemma}\label{L:primed} The numbers $k$ and $\gcd(a,b)$ have the same prime divisors.
\end{lemma}

\begin{proof}
From \eqref{E:xbc},
$a^2-kmn a+k(m^2+n^2)=0$ and $b^2-kmn b+k(m^2+n^2)=0$.  So if $p$ is any prime divisor of $k$, then $p$ divides $a^2$ and $b^2$, so $p$ divides $a$ and $b$. Hence $p$ divides  $\gcd(a,b)$. 

Conversely, let $p$ be any prime divisor of $\gcd(a,b)$, and suppose that $p$ does not divide $k$. Then $p^2$ divides $m^2+n^2$ and $p$ divides $mn$ so $p$ divides $(m+n)^2$ and $(m-n)^2$, so $p$ divides $m+n$ and $m-n$. Hence $p$ divide $2m$ and $2n$. If $p$ were odd we would have that $p$ divides $m$ and $n$, which is impossible as $m,n$ are relatively prime. So $p=2$; that is, $a,b$ are both even. Since $p$ does not divide $k$, we have that $k$ is odd and $m^2+n^2$ is even. So, as $m,n$ are relatively prime, $m,n$ are both odd. But then $a+b=kmn$ implies that $k$ is even, a contradiction. So $p$ divides $k$.
\end{proof}

Now suppose that $(a,b)\in \mathcal T_3$. So $(a,b)=(3q, 3r)$, where $m^2 + n^2 + q^2 = 3mnq$ and $r=3mn-q$. By the above lemma,  $\gcd(a,b)$ is a power of 3. Assume that $q$ is divisible by 3. Then $m^2 + n^2\equiv 0\pmod 3$ and hence $m,n\equiv 0\pmod 3$. But this contradicts the fact that
$m,n$ are relatively prime. So $\gcd(a,b)=3$.

Similarly, suppose that $(a,b)\in \mathcal T_4$. So $(a,b)=(4q, 4r)$, where $m^2 + n^2 + 2q^2 = 4mnq$. By the above lemma,  $\gcd(a,b)$ is a power of 2. Assume that $q$ is even. Then $m^2 + n^2\equiv 0\pmod 4$ and hence $m,n\equiv 0\pmod 4$, contradicting the fact that
$m,n$ are relatively prime. So $\gcd(a,b)=4$.

Finally, suppose that $(a,b)\in \mathcal T_5$. So $(a,b)=(5q, 5r)$, where $m^2 + n^2 + 5q^2 = 5mnq$ and $r =mn-q$. By the above lemma,  $\gcd(a,b)$ is a power of 5. Assume that $q,r$ are both divisible by 5. Then $mn=q+r \equiv 0\pmod 5$, so either $m$ or $n$ is divisible by 5. But we also have $m^2 + n^2\equiv 0\pmod 5$ and hence $m$ and $n$ are both divisible by 5. But this again contradicts the fact that
$m,n$ are relatively prime. So $\gcd(a,b)=5$.

This completes the proof of Theorem \ref{T:345}.\hfill\qed

\begin{remark}\label{R:div}
 In the case $\gcd(a,b)=3$, it is evident from the above proof that $a+b$ is divisible by $3^2$, but   neither $a$ nor $b$ is divisible by $3^2$. Similarly, if $\gcd(a,b)=4$,  then $a+b$ is divisible by $2^3$ but  neither $a$ nor $b$ is divisible by $2^3$.  However, a similar result does not hold in the $\gcd(a,b)=5$ case. For example, for the LEP with $a=85,b=  1525$ one has  $a=5\cdot 17, b=5^2 \cdot 61$. There is also an obvious restriction on the prime divisors of $a$ or $b$. Suppose that $(a,b)\in \mathcal T_3$, so $(a,b)=(3q, 3r)$ as before. We saw at the beginning of this section that 
 $qr=m^2+n^2$, a sum of two squares. Hence $q$ and $r$ have no prime divisor congruent to $3 \pmod 4$.
 By the same reasoning, if $(a,b)$ belongs to $\mathcal T_4$ or $\mathcal T_5$, then $a$ and $b$ have no prime divisor congruent to $3 \pmod 4$.

\end{remark}

%%%%%%%%%%%%%%%%%%%%%%%%%

\section{The forest of LEPs}\label{S:forest}

For a given $k$ ($=5,8$ or $9$) consider the set $S_k$ of solutions of the corresponding Markov-Rosenberger equation. Let us briefly recall Rosenberger's theory. For the reader's convenience, we recall \eqref{E:xb3}:
\[
m^2+n^2+f\g^2=f\!s\, mn\g.
\]
Following the presentation given in \cite{BU}, from a solution $x = (m,n,\g)$ to \eqref{E:xb3}, one can generate three new solutions by
applying the involutions:
\begin{align*}
\phi_1(x)&= (fs n\g-m,n,\g)\\
\phi_2(x)&= (m,fsm\g-n,\g)\\
\phi_3(x)&= (m,n,smn-\g) .
\end{align*}
The group of transformations of $S_k$ generated by the maps $\phi_i$ is the free product of three copies of $\Z_2$, and this group acts transitively on $S_k$.
Moreover, the  maps $\phi_i$ give the set $S_k$ of solutions the structure of an infinite  binary tree: each solution is a vertex and two solutions are connected by an edge if  one of the  maps $\phi_i$ sends one solution to the other.
The \emph{fundamental solutions} have the smallest values of $m+n+\g$; in M,R1,R3 they are $(m,n,\g)=(1,1,1),(1,1,1),(1,2,1)$ respectively. 

Having recalled Rosenberger's theory, we now describe how the solution trees $S_k$ of the Markov-Rosenberger equations determine the induced structure on the set $\mathcal T_{gcd}$ of  LEPs, where $gcd=3,4,5$ for $k=9,8,5$ respectively. We identify $\mathcal T_{gcd}$ as the set of ordered pairs $(a,b)$ of possible side lengths  with $\gcd(a,b)={gcd}$ and $a\le b$. A solution $(m,n,\g)$ to the corresponding Markov-Rosenberger equation corresponds to a LEP if  $n\ge m$. Let $\tau$ denote the involution of $S_k$ given by $\tau:(m,n,\g)\mapsto (n,m,\g)$. The map $\tau$ commutes with $\phi_3$, and conjugates $\phi_1$ to $\phi_2$. We form the quotient tree $\bar S_k=S_k/\tau$. The elements of $\bar S_k$ can be regarded as triples $(m,n,\g)$ with $m\le n$.
The map $\phi_3$ induces an involution $\bar \phi_3$ on $\bar S_k$. Note that by definition $ \phi_3$ leaves $m$ and $n$  unchanged and sends $\g$ to $smn-\g$, or equivalently, $a$ to $kmn-a=a+b-a=b$. Thus $(m,n,\g)$ and $ \phi_3(m,n,\g)$ correspond to the same LEP. Consequently we form $\mathcal T_{gcd}$ by contracting each of the edges of $\bar S_k$ that are given by the map $\bar \phi_3$; see Figure \ref{F:contract}. Note that contraction of edges of a tree produces another tree. So $\mathcal T_{gcd}$ is a tree for each ${gcd}=3,4,5$. Notice also that contraction turns vertices of degree 3 into vertices of degree 4. The trees $\mathcal T_3$ and $\mathcal T_5$ have the same form; they are shown in Figure \ref{F:tree35}. The tree $\mathcal T_4$  is shown in Figure \ref{F:tree4}. The \emph{fundamental solutions} are the solutions $(a,b)$ for which $a+b$ is minimal. For $\gcd(a,b)=3,4,5$, the fundamental solutions are $(3,6),(4,4),(5,5)$ respectively.
Apart from the fundamental solutions, the vertices of $\mathcal T_{gcd}$ all have degree 4. The fundamental solutions of $\mathcal T_3$ and $\mathcal T_5$ have degree 2, while the fundamental solution  of $\mathcal T_4$ has degree  1.

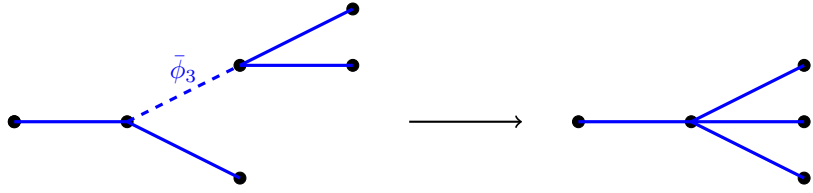
\begin{figure}[h]
\begin{tikzpicture}[scale=1.5][line cap=round,line join=round,>=triangle 45,x=1cm,y=.5cm]
\def\y{.5};
\def\d{5};
\draw [fill=black] (-1,0) circle (1.5pt);
\draw [fill=black] (0,0) circle (1.5pt);
\draw [fill=black] (1,\y) circle (1.5pt);
\draw [fill=black] (1,-\y) circle (1.5pt);
\draw [fill=black] (-1,0) circle (1.5pt);
\draw [fill=black] (2,2*\y) circle (1.5pt);
\draw [fill=black] (2,\y) circle (1.5pt);
\draw [line width=1.2pt,dashed,color=qqqqff] (0,0)-- (1,\y) node[midway,above] {$\bar\phi_3$};
\draw [line width=1.2pt,color=qqqqff] (0,0)-- (1,-\y) ;
\draw [line width=1.2pt,color=qqqqff] (0,0)-- (-1,0) ;
\draw [line width=1.2pt,color=qqqqff] (1,\y)-- (2,2*\y) ;
\draw [line width=1.2pt,color=qqqqff] (1,\y)-- (2,\y) ;
\draw [->,line width=0.8pt] (2.5,0) -- (3.5,0);
draw [fill=black] (\d-1,0) circle (1.5pt);
\draw [fill=black] (\d,0) circle (1.5pt);
\draw [fill=black] (\d+1,-\y) circle (1.5pt);
\draw [fill=black] (\d-1,0) circle (1.5pt);
\draw [fill=black] (\d+1,\y) circle (1.5pt);
\draw [fill=black] (\d+1,0) circle (1.5pt);
\draw [line width=1.2pt,color=qqqqff] (\d,0)-- (\d+1,-\y) ;
\draw [line width=1.2pt,color=qqqqff] (\d,0)-- (\d-1,0) ;
\draw [line width=1.2pt,color=qqqqff] (\d,0)-- (\d+1,\y) ;
\draw [line width=1.2pt,color=qqqqff] (\d,0)-- (\d+1,0) ;
\end{tikzpicture}
\caption{Contraction of the $\bar \phi_3$ edges}\label{F:contract}
\end{figure}

Let us look in more detail at the edges of the LEP solution trees $\mathcal T_{gcd}$.
Consider a LEP $(a,b)$. As we saw in the previous section, $k$ is determined by $\gcd(a,b)$; $k=9,8,5$ for $\gcd(a,b)=3,4,5$ respectively.   
Writing $ab= k(m^2+n^2), a+b=kmn$ as before, and imposing $n\ge m$, the values $m,n$ are given by the following formulae:
\begin{align}
kn^2&=\frac{ab+\sqrt{a^2 b^2 -4(a+b)^2}}{2}\label{E:n}\\
km^2&=\frac{ab-\sqrt{a^2 b^2 -4(a+b)^2}}{2}.\label{E:m}
\end{align}
For each $i=1,2$, let $\v_i$ denote the function induced by $\phi_i$ on the set $\mathcal T_{gcd}$ of LEP pairs $(a,b)$. Note that in the above notation, as $a+b=kmn$,
we have $(a,b)=(fs\g, kmn-fs\g)$.
From the definition of $\phi_1$, the  map $\v_1$ leaves $a$ and $n$  unchanged and $m$ is changed to $m'= an-m$. Then under $\v_1$, the value of $b$ is changed to
\begin{align*}
b'&= km'n-a=kan^2-knm-a=kan^2-(a+b)-a\\
&=\frac{a^2b+ a \sqrt{a^2b^2-4(a+b)^2}}2-2a-b. \quad  (\text{from}\  \eqref{E:n})
\end{align*}
Thus
\[
\v_1: (a,b) \mapsto \left(a,\frac{a^2b+ a \sqrt{a^2b^2-4(a+b)^2}}2-2a-b\right).
\]
Similarly, one finds that
\[
\v_2: (a,b) \mapsto \left(a,\frac{a^2b- a \sqrt{a^2b^2-4(a+b)^2}}2-2a-b\right).
\]
A calculation shows that  $\v_2\circ\v_1=\id$. However it is not true that $\v_1\circ\v_2(a,b)=(a,b)$ for all $(a,b)$. In fact, the function $\v_2$ is not injective and the function $\v_1$ is not surjective; there is no solution $(a,b)$ with $a<b$ for which $\v_1(a,b)$ is the fundamental solution.

Similarly, analogous to $\v_1,\v_2$, interchanging the roles of $a$ and $b$, we have two further maps:
\begin{align*}
\psi_1: (a,b) &\mapsto \left(\frac{ab^2+b\sqrt{a^2b^2-4(a+b)^2}}2-2b-a,b\right),\\
\psi_2: (a,b) &\mapsto \left(\frac{ab^2-b\sqrt{a^2b^2-4(a+b)^2}}2-2b-a,b\right).
\end{align*}
After applying the maps $\v_i,\psi_i$, one reverses the image $(a',b')$ if necessary so that $a'\le b'$. The edges in the trees $\mathcal T_{gcd}$ are all obtained by applications of the four maps $\v_1,\v_2,\psi_1,\psi_2$.

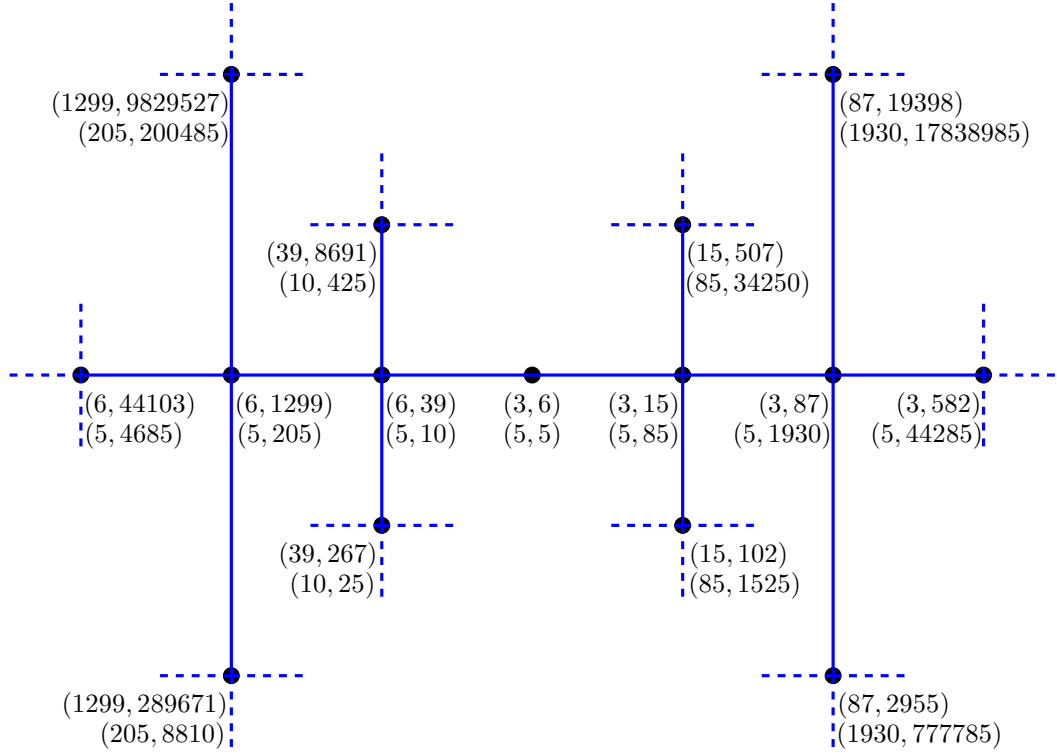
\begin{figure}[h]
\begin{tikzpicture}[scale=2][line cap=round,line join=round,>=triangle 45,x=1cm,y=1cm]
\def\fx{0};
\def\fy{0};
\def\rx{1};
\def\ry{0};
\def\rrx{2};
\def\rry{0};
\def\rrrx{3};
\def\rrry{0};
\def\rux{1};
\def\ruy{1};
\def\rdx{1};
\def\rdy{-1};
\def\rrux{2};
\def\rruy{2};
\def\rrdx{2};
\def\rrdy{-2};
\def\lx{-1};
\def\ly{0};
\def\lux{-1};
\def\luy{1};
\def\ldx{-1};
\def\ldy{-1};
\def\llx{-2};
\def\lly{0};
\def\llux{-2};
\def\lluy{2};
\def\lldx{-2};
\def\lldy{-2};
\def\lllx{-3};
\def\llly{0};
\draw [fill=black] (\fx,\fy) circle (1.5pt);
\draw[color=black] (\fx,\fy-.2) node {$(3,6)$};
\draw[color=black] (\fx,\fy-.4) node {$(5,5)$};
\draw [fill=black] (\rx,\ry) circle (1.5pt);
\draw[color=black] (\rx-.26,\ry-.2) node {$(3,15)$};
\draw[color=black] (\rx-.26,\ry-.4) node {$(5,85)$};
\draw [fill=black] (\lx,\ly) circle (1.5pt);
\draw[color=black] (\lx+.26,\ly-.2) node {$(6,39)$};
\draw[color=black] (\lx+.26,\ly-.4) node {$(5,10)$};
\draw [fill=black] (\rrx,\rry) circle (1.5pt);
\draw[color=black] (\rrx-.26,\rry-.2) node {$(3,87)$};
\draw[color=black] (\rrx-.34,\rry-.4) node {$(5,1930)$};
\draw [fill=black] (\rrrx,\rrry) circle (1.5pt);
\draw[color=black] (\rrrx-.3,\rrry-.2) node {$(3,582)$};
\draw[color=black] (\rrrx-.38,\rrry-.4) node {$(5,44285)$};
\draw [fill=black] (\llx,\lly) circle (1.5pt);
\draw[color=black] (\llx+.35,\lly-.2) node {$(6,1299)$};
\draw[color=black] (\llx+.32,\lly-.4) node {$(5,205)$};
\draw [fill=black] (\rux,\ruy) circle (1.5pt);
\draw[color=black] (\rux+.35,\ruy-.2) node {$(15,507)$};
\draw[color=black] (\rux+.43,\ruy-.4) node {$(85,34250)$};
\draw [fill=black] (\rdx,\rdy) circle (1.5pt);
\draw[color=black] (\rdx+.37,\rdy-.2) node {$(15,102)$};
\draw[color=black] (\rdx+.41,\rdy-.4) node {$(85,1525)$};
\draw [fill=black] (\rrux,\rruy) circle (1.5pt);
\draw[color=black] (\rrux+.45,\rruy-.2) node {$(87,19398)$};
\draw[color=black] (\rrux+.67,\rruy-.4) node {$(1930,17838985)$};
\draw [fill=black] (\rrdx,\rrdy) circle (1.5pt);
\draw[color=black] (\rrdx+.4,\rrdy-.2) node {$(87,2955)$};
\draw[color=black] (\rrdx+.57,\rrdy-.4) node {$(1930,777785)$};
\draw [fill=black] (\lux,\luy) circle (1.5pt);
\draw[color=black] (\lux-.4,\luy-.2) node {$(39,8691)$};
\draw[color=black] (\lux-.36,\luy-.4) node {$(10,425)$};
\draw [fill=black] (\ldx,\ldy) circle (1.5pt);
\draw[color=black] (\ldx-.36,\ldy-.2) node {$(39,267)$};
\draw[color=black] (\ldx-.33,\ldy-.4) node {$(10,25)$};
\draw [fill=black] (\lllx,\llly) circle (1.5pt);
\draw[color=black] (\lllx+.39,\llly-.2) node {$(6,44103)$};
\draw[color=black] (\lllx+.35,\llly-.4) node {$(5,4685)$};
\draw [fill=black] (\llux,\lluy) circle (1.5pt);
\draw[color=black] (\llux-.61,\lluy-.2) node {$(1299,9829527)$};
\draw[color=black] (\llux-.52,\lluy-.4) node {$(205,200485)$};
\draw [fill=black] (\lldx,\lldy) circle (1.5pt);
\draw[color=black] (\lldx-.58,\lldy-.2) node {$(1299,289671)$};
\draw[color=black] (\lldx-.46,\lldy-.4) node {$(205,8810)$};
\draw [line width=1.2pt,color=qqqqff] (\fx,\fy)-- (\rx,\ry);
\draw [line width=1.2pt,color=qqqqff] (\rx,\ry)-- (\rrx,\rry);
\draw [line width=1.2pt,color=qqqqff] (\fx,\fy)-- (\lx,\ly);
\draw [line width=1.2pt,color=qqqqff] (\rx,\ry)-- (\rux,\ruy);
\draw [line width=1.2pt,color=qqqqff] (\rx,\ry)-- (\rdx,\rdy);
\draw [line width=1.2pt,color=qqqqff] (\rrx,\rry)-- (\rrrx,\rrry);
\draw [line width=1.2pt,color=qqqqff] (\rrx,\rry)-- (\rrux,\rruy);
\draw [line width=1.2pt,color=qqqqff] (\rrx,\rry)-- (\rrdx,\rrdy);
\draw [line width=1.2pt,color=qqqqff] (\lx,\ly)-- (\llx,\lly);
\draw [line width=1.2pt,color=qqqqff] (\lx,\ly)-- (\lux,\luy);
\draw [line width=1.2pt,color=qqqqff] (\lx,\ly)-- (\ldx,\ldy);
\draw [line width=1.2pt,color=qqqqff] (\lllx,\llly)-- (\llx,\lly);
\draw [line width=1.2pt,color=qqqqff] (\llux,\lluy)-- (\llx,\lly);
\draw [line width=1.2pt,color=qqqqff] (\lldx,\lldy)-- (\llx,\lly);
\draw [line width=1.2pt,dashed,color=qqqqff] (\rux,\ruy)-- (\rux+.5,\ruy);
\draw [line width=1.2pt,dashed,color=qqqqff] (\rux,\ruy)-- (\rux,\ruy+.5);
\draw [line width=1.2pt,dashed,color=qqqqff] (\rux,\ruy)-- (\rux-.5,\ruy);
\draw [line width=1.2pt,dashed,color=qqqqff] (\rdx,\rdy)-- (\rdx+.5,\rdy);
\draw [line width=1.2pt,dashed,color=qqqqff] (\rdx,\rdy)-- (\rdx,\rdy-.5);
\draw [line width=1.2pt,dashed,color=qqqqff] (\rdx,\rdy)-- (\rdx-.5,\rdy);
\draw [line width=1.2pt,dashed,color=qqqqff] (\rrux,\rruy)-- (\rrux+.5,\rruy);
\draw [line width=1.2pt,dashed,color=qqqqff] (\rrux,\rruy)-- (\rrux,\rruy+.5);
\draw [line width=1.2pt,dashed,color=qqqqff] (\rrux,\rruy)-- (\rrux-.5,\rruy);
\draw [line width=1.2pt,dashed,color=qqqqff] (\rrdx,\rrdy)-- (\rrdx+.5,\rrdy);
\draw [line width=1.2pt,dashed,color=qqqqff] (\rrdx,\rrdy)-- (\rrdx,\rrdy-.5);
\draw [line width=1.2pt,dashed,color=qqqqff] (\rrdx,\rrdy)-- (\rrdx-.5,\rrdy);

\draw [line width=1.2pt,dashed,color=qqqqff] (\rrrx,\rrry)-- (\rrrx+.5,\rrry);
\draw [line width=1.2pt,dashed,color=qqqqff] (\rrrx,\rrry)-- (\rrrx,\rrry-.5);
\draw [line width=1.2pt,dashed,color=qqqqff] (\rrrx,\rrry)-- (\rrrx,\rrry+.5);

\draw [line width=1.2pt,dashed,color=qqqqff] (\lux,\luy)-- (\lux+.5,\luy);
\draw [line width=1.2pt,dashed,color=qqqqff] (\lux,\luy)-- (\lux,\luy+.5);
\draw [line width=1.2pt,dashed,color=qqqqff] (\lux,\luy)-- (\lux-.5,\luy);
\draw [line width=1.2pt,dashed,color=qqqqff] (\ldx,\ldy)-- (\ldx+.5,\ldy);
\draw [line width=1.2pt,dashed,color=qqqqff] (\ldx,\ldy)-- (\ldx,\ldy-.5);
\draw [line width=1.2pt,dashed,color=qqqqff] (\ldx,\ldy)-- (\ldx-.5,\ldy);
\draw [line width=1.2pt,dashed,color=qqqqff] (\llux,\lluy)-- (\llux+.5,\lluy);
\draw [line width=1.2pt,dashed,color=qqqqff] (\llux,\lluy)-- (\llux,\lluy+.5);
\draw [line width=1.2pt,dashed,color=qqqqff] (\llux,\lluy)-- (\llux-.5,\lluy);
\draw [line width=1.2pt,dashed,color=qqqqff] (\lldx,\lldy)-- (\lldx+.5,\lldy);
\draw [line width=1.2pt,dashed,color=qqqqff] (\lldx,\lldy)-- (\lldx,\lldy-.5);
\draw [line width=1.2pt,dashed,color=qqqqff] (\lldx,\lldy)-- (\lldx-.5,\lldy);

\draw [line width=1.2pt,dashed,color=qqqqff] (\lllx,\llly)-- (\lllx-.5,\llly);
\draw [line width=1.2pt,dashed,color=qqqqff] (\lllx,\llly)-- (\lllx,\llly-.5);
\draw [line width=1.2pt,dashed,color=qqqqff] (\lllx,\llly)-- (\lllx,\llly+.5);
\end{tikzpicture}
\caption{The two trees of LEPs with $\gcd(a,b)=3$ and $5$ respectively}\label{F:tree35}
\end{figure}

\begin{figure}[h]
\begin{tikzpicture}[scale=2][line cap=round,line join=round,>=triangle 45,x=1cm,y=1cm]
\def\fx{0};
\def\fy{0};
\def\rx{1};
\def\ry{0};
\def\rrx{2};
\def\rry{0};
\def\rrrx{3};
\def\rrry{0};
\def\rux{1};
\def\ruy{1};
\def\rdx{1};
\def\rdy{-1};
\def\rrux{2};
\def\rruy{2};
\def\rrdx{2};
\def\rrdy{-2};
\draw [fill=black] (\fx,\fy) circle (1.5pt);
\draw[color=black] (\fx,\fy-.2)  node {$(4,4)$};
\draw [fill=black] (\rx,\ry) circle (1.5pt);
\draw[color=black] (\rx-.26,\ry-.2) node {$(4,20)$};
\draw [fill=black] (\rrx,\rry) circle (1.5pt);
\draw[color=black] (\rrx-.3,\rry-.2) node {$(4,260)$};
\draw [fill=black] (\rrrx,\rrry) circle (1.5pt);
\draw[color=black] (\rrrx-.35,\rrry-.2) node {$(4,3604)$};
\draw [fill=black] (\rux,\ruy) circle (1.5pt);
\draw[color=black] (\rux+.4,\ruy-.2) node {$(20,1396)$};
\draw [fill=black] (\rdx,\rdy) circle (1.5pt);
\draw[color=black] (\rdx+.37,\rdy-.2) node {$(20,116)$};
\draw [fill=black] (\rrux,\rruy) circle (1.5pt);
\draw[color=black] (\rrux+.55,\rruy-.2) node {$(260,251156)$};
\draw [fill=black] (\rrdx,\rrdy) circle (1.5pt);
\draw[color=black] (\rrdx+.5,\rrdy-.2) node {$(260,18196)$};
\draw [line width=1.2pt,color=qqqqff] (\fx,\fy)-- (\rx,\ry);
\draw [line width=1.2pt,color=qqqqff] (\rx,\ry)-- (\rrx,\rry);
\draw [line width=1.2pt,color=qqqqff] (\rx,\ry)-- (\rux,\ruy);
\draw [line width=1.2pt,color=qqqqff] (\rx,\ry)-- (\rdx,\rdy);
\draw [line width=1.2pt,color=qqqqff] (\rrx,\rry)-- (\rrrx,\rrry);
\draw [line width=1.2pt,color=qqqqff] (\rrx,\rry)-- (\rrux,\rruy);
\draw [line width=1.2pt,color=qqqqff] (\rrx,\rry)-- (\rrdx,\rrdy);
\draw [line width=1.2pt,dashed,color=qqqqff] (\rux,\ruy)-- (\rux+.5,\ruy);
\draw [line width=1.2pt,dashed,color=qqqqff] (\rux,\ruy)-- (\rux,\ruy+.5);
\draw [line width=1.2pt,dashed,color=qqqqff] (\rux,\ruy)-- (\rux-.5,\ruy);
\draw [line width=1.2pt,dashed,color=qqqqff] (\rdx,\rdy)-- (\rdx+.5,\rdy);
\draw [line width=1.2pt,dashed,color=qqqqff] (\rdx,\rdy)-- (\rdx,\rdy-.5);
\draw [line width=1.2pt,dashed,color=qqqqff] (\rdx,\rdy)-- (\rdx-.5,\rdy);
\draw [line width=1.2pt,dashed,color=qqqqff] (\rrux,\rruy)-- (\rrux+.5,\rruy);
\draw [line width=1.2pt,dashed,color=qqqqff] (\rrux,\rruy)-- (\rrux,\rruy+.5);
\draw [line width=1.2pt,dashed,color=qqqqff] (\rrux,\rruy)-- (\rrux-.5,\rruy);
\draw [line width=1.2pt,dashed,color=qqqqff] (\rrdx,\rrdy)-- (\rrdx+.5,\rrdy);
\draw [line width=1.2pt,dashed,color=qqqqff] (\rrdx,\rrdy)-- (\rrdx,\rrdy-.5);
\draw [line width=1.2pt,dashed,color=qqqqff] (\rrdx,\rrdy)-- (\rrdx-.5,\rrdy);

\draw [line width=1.2pt,dashed,color=qqqqff] (\rrrx,\rrry)-- (\rrrx+.5,\rrry);
\draw [line width=1.2pt,dashed,color=qqqqff] (\rrrx,\rrry)-- (\rrrx,\rrry-.5);
\draw [line width=1.2pt,dashed,color=qqqqff] (\rrrx,\rrry)-- (\rrrx,\rrry+.5);

\end{tikzpicture}
\caption{The  tree of LEPs with $\gcd(a,b)=4$}\label{F:tree4}
\end{figure}
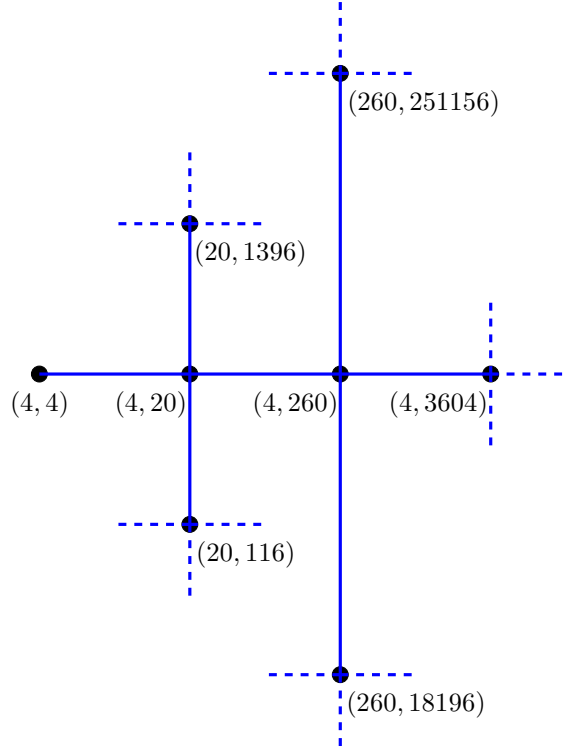

%%%%%%%%%%%%%%%%%%%%%%%%%

\section{Recurrence relations in the branches of the LEP trees}\label{S:rec}

The trees $\mathcal T_3,\mathcal T_4,\mathcal T_5$ contain a raft of interesting patterns that connect with phenomena that have been observed in other areas. We will give some simple examples.
First we require a technical lemma concerning recurrence relations. Consider the real function 
\[
f(x)= u x + v +  \sqrt{(u^2 - 1)  x^2 + 2 (u + 1) v x + w},
\]
 where $u,v,w$ are reals with $u>1$. Notice that for large $x$ the function $f(x)$ is asymptotic to the linear function $(u +  \sqrt{(u^2 - 1) )} x+v$, which is monotonically increasing as $u>1$.  If necessary, we restrict the domain of $f(x)$ to a maximal half-infinite interval  $[a,\infty)$ on which $f(x)$ is defined,  and has a well defined inverse.

\begin{lemma}\label{L:recu} Consider the function 
$f(x)= u x + v +  \sqrt{(u^2 - 1)  x^2 + 2 (u + 1) v x + w}$ 
 with the $u>1$. Choose $a_0$ in the domain of $f(x)$ as described above and define $a_i$ recursively by $a_i=f(a_{i-1})$. Then $a_i$ satisfies the 3rd order recurrence relation 
\[
a_i=(2u+1) a_{i-1} - (2u+1)a_{i-2} + a_{i-3}.
\]
\end{lemma}

\begin{remark}\label{R:affine}
It is not difficult to see that if a function $f(x)$ is of the form required by Lemma~\ref{L:recu}, then for any affine change of variable, $X=rx+s$, the function induced  by $f$ on $X$ also has the same form, with the same value of $u$, while $v$ changes to $rv-us+s$. In particular, as $u>1$, one can choose $s=rv/(u-1)$ to make $v$ become $0$.
\end{remark}

\begin{comment}
\begin{align*}
X &-> r(u x + v +  \sqrt{(u^2 - 1)  x^2 + 2 (u + 1) v x + w})+s\\
&=r(u (X-s)/r + v +  \sqrt{(u^2 - 1)  ( (X-s)/r)^2 + 2 (u + 1) v  (X-s)/r + w})+s\\
&= uX -us+s +rv + \sqrt{(u^2 - 1)   (X-s)^2 + 2 (u + 1) v  (X-s)r + wr^2}\\
&= uX -us+s +rv + \sqrt{(u^2 - 1)   X^2 -2s(u^2 - 1)X + 2 (u + 1) rv  (X-s) + wr^2+(u^2 - 1) s^2}\\
&= uX -us+s +rv + \sqrt{(u^2 - 1)   X^2 +(2 (u + 1) rv-2s(u^2 - 1))X + const}
\end{align*}
So we need to see that 
\[
2 (u + 1) rv-2s(u^2 - 1)=2(u+1)(-us+s +rv )
\]
but this is obvious.
\end{comment}

\begin{proof}[Proof of Lemma \ref{L:recu}] The required recurrence relation is clearly invariant under affine transformations. So, by the above remark, we may suppose that $v=0$. Thus
\[
f(x)= u x +  \sqrt{(u^2 - 1)  x^2 + w}.
\]
We start by showing that 
%$f^{-1}(x)= (2u+1) -  \sqrt{(u^2 - 1)  x^2 + 2 (u + 1) v x + w}$
$f^{-1}(x)= u x  -  \sqrt{(u^2 - 1)  x^2 + w}$. Indeed, if $f(y)=x$, then
$u y +   \sqrt{(u^2 - 1)  y^2 + w} =x$.
 Rearranging and squaring gives
$  y^2 - 2 ux y  +x^2-w=0$,
and solving for $y$ gives
$y= ux  \pm\sqrt{(u^2-1)x^2+w}$.
As $u>1$, we have $f(x)>x$ for sufficiently large $x$, so $f^{-1}(x)<x$ and we take the negative root in the  formula for $y$. So $f^{-1}(x)$ is as claimed.

Let $E:=f(f(x))-((2u+1) f(x) - (2u+1)x + f^{-1}(x))$. We claim that $E=0$. 
Indeed, let 
$D=(u^2 - 1)  x^2 + w$
  and note that
\[
f(f(x))=  u^2x +   u\sqrt{D}  +  \sqrt{(u^2-1)( u x  +  \sqrt{D})^2 +   w}\]
and
\[
(2u+1) f(x) - (2u+1)x + f^{-1}(x)=(2 u^2 -1)x +   2 u \sqrt{D}.
\]
So  $E$ simplifies to
\[
-[(u^2-1) x +u \sqrt{D}] + \sqrt{(u^2-1)( u x  +  \sqrt{D})^2 +   w}.
\]
Note that as $u>1$, we have $(u^2-1) x +u \sqrt{D}>0$. Hence to show that $E=0$ we need to see that
\[
[(u^2-1) x +u \sqrt{D}]^2-[(u^2-1)( u x  +  \sqrt{D})^2 +   w]=0.
\]
Expanding and simplifying the left-hand-side we are left with
\[
%u^2-1)^2 x^2+u^2D +&2u(u^2-1) x  \sqrt{D}- u^2(u^2-1)x^2-(u^2-1)D-2u(u^2-1)x\sqrt{D}-w\\
-(u^2-1) x^2+D-w,
\]
which is 0, as claimed, by the definition of $D$.

As $E=0$ we have $f(f(x))=(2u+1) f(x) - (2u+1)x + f^{-1}(x)$, for all $x$ in the appropriate domain. Replacing $x$ by $f(x)$ gives
\[
f(f(f(x)))=(2u+1) f(f(x)) - (2u+1)f(x) + x.
\]
Applying this equation to $x=a_0$ gives the required recurrence relation. 
\end{proof}

For the function $\v_1 (a,b) = (a,\frac{a^2b+ a \sqrt{a^2b^2-4(a+b)^2}}2-2a-b)$ defined in the previous section, let us fix $a$ and look at the resulting function of $b$:
\begin{equation}\label{E:c}
f: b \mapsto \frac{a^2b+a\sqrt{a^2b^2-4(a+b)^2}}2-2a-b.
\end{equation}
One readily verifies that this function $f$ has the form required by Lemma \ref{L:recu}, with $u=\frac{ a^2-2}2,v=-2 a ,w=- a^4$.
Then  Lemma \ref{L:recu} gives the recurrence relation 
\begin{equation}\label{E:crel}
b_i=(a^2-1) b_{i-1} - (a^2-1)b_{i-2} + b_{i-3}.
\end{equation}

\begin{remark}
While the function $f$ of \eqref{E:c} is useful in its own right, we now briefly outline a  proof of the recurrence relation 
\eqref{E:crel} that does not require Lemma \ref{L:recu}. As we saw in the previous section, the  map $\v_1$ leaves $a$ and $n$  unchanged and $m$ is changed to $m'= an-m$. But this gives $m'>n$ so we have to then interchange $m'$ and $n$. Thus we can write $n'=an-m$ and $m'=n$. It follows that under two applications, we have $m'' =n'=an-m=am'-m$. Consequently, the value $m$ satisfies the second order recurrence relation $m_{i+1}=am_i-m_{i-1}$. Similarly, $n_{i+1}=an_i-n_{i-1}$. According to  a well known theorem of E. S. Selmer (see \cite{BS})  if one has a  second order recurrence relation $y_{i+1} = Ay_{i}+By_{i-1}$  and $x^2-Ax-B = (x-\alpha)(x-\beta)$ with $\alpha\not=\beta$, then $y_i^2$ satisfies the third order recurrence relation $y^2_{i+1} = Cy^2_{i}+Dy^2_{i-1}+Ey^2_{i-2} $,  where 
$x^3 -Cx^2 - Dx- E = (x- \alpha^2)(x- \beta^2)(x-\alpha\beta)$. In particular, if $y_{i+1} = a y_{i}-y_{i-1}$ with $a\ge 3$, then $y^2_{i+1} = (a^2-1)y^2_{i}-(a^2-1)y^2_{i-1}+y^2_{i-2}$.
So  the sequences for $m^2$ and $n^2$ both satisfy this third order recurrence relation, and hence so too does $b=k(m^2+n^2)/a$.
\end{remark}

We now examine the horizontal branches of the trees $\mathcal T_3,\mathcal T_4,\mathcal T_5$ that start at the fundamental solutions and head towards the right. These solutions have constant $a$ and are defined by repeated applications of the map $\v_1$, or equivalently, repeated applications of the function $f$ defined in \eqref{E:c}.

\begin{example}\label{Eg:3}
For $a=3$, with the above notation, \eqref{E:c} gives
\[
f: b \mapsto \frac{7}2b-6+\sqrt{\frac{45}4 b^2-54b-81}.
\]
Setting $b_0= 6$, corresponding to the fundamental LEP $a=3,b=6$, the first few terms of the sequence are
$6, 15, 87, 582, 3975$. The first 4 of these are visible in Figure \ref{F:tree35}; they lie on the central horizontal branch to the right of the fundamental solution $(3,6)$.
We employ Lemma \ref{L:recu} with $u=\frac{7}2, v=-6, w=-81$.
Thus
\[
b_i=8 b_{i-1} - 8b_{i-2} + b_{i-3},\quad\text{for}\ i\ge 3.
\]
%582=8 x 87 -8x 15+6
%3975=8x582-887+15
By Remark \ref{R:div}, the area of the LEP, $2(a+b)$, is divisible by 9, and it is obviously even.
Let $A=2(a+b)/18=(3+b)/9$.
This is an affine transformation of $b$, so by Remark~\ref{R:affine}, the function induced by $f$ on the variable $A$ also has the form required by Lemma \ref{L:recu} with the same value of $u$. So by
Lemma \ref{L:recu},
\[
A_i=8 A_{i-1} - 8A_{i-2} + A_{i-3},\quad\text{for}\ i\ge 3.
\]
Setting $A_0= 2$, corresponding to the fundamental LEP $a=3,b=6$, the first few terms of the sequence are
$1,2,10,65,442$. This sequence is related to the sequence A064170, in  OEIS \cite{OEIS}, which is conjectured to satisfy the same recurrence relation.
\end{example}

\begin{example}\label{Eg:4}
For  $a=4$, \eqref{E:c} gives
\[
f: b \mapsto 7b-8+4\sqrt{3b^2-8b-16}.
\]
Setting $b_0= 4$, corresponding to the fundamental LEP $a=4,b=4$, the first few terms of the sequence are
$4, 20,260,3604,50180$. The first 4 of these are visible in Figure \ref{F:tree4}; they lie on the central horizontal branch to the right of the fundamental solution $(4,4)$.

We use Lemma \ref{L:recu} with $u=7, v=-8, w=-256$.
% -8 16 =2(7+1)(-8)
Thus
\[
b_i=15 b_{i-1} - 15b_{i-2} + b_{i-3},\quad\text{for}\ i\ge 3.
\]
Dividing the $b$-values by 4 we have the sequence $1,5,65,901,12545,\dots$. This is the sequence A103974 in  OEIS \cite{OEIS}; it is the sequence of smaller sides $x$ in $(x,x,x+1)$-integer triangles with integer area.

Further, by Remark \ref{R:div}, the area, $2(a+b)$ is divisible by 16.
Let $A=(a+b)/8=(4+b)/8$. As in the previous example, this is an affine transformation of $b$, so by Remark~\ref{R:affine}, the function induced by $f$ on the variable $A$ also has the form required by Lemma \ref{L:recu} with the same value of $u$.
Thus
\[
A_i=15 A_{i-1} - 15A_{i-2} + A_{i-3},\quad\text{for}\ i\ge 3.
\]
Setting $A_0= 1$, corresponding to the fundamental LEP $a=4,b=4$, the first few terms of the sequence are
$1,3,33,451,6273$. This is the sequence A011922 in  OEIS \cite{OEIS}.
\end{example}

\begin{example}\label{Eg:5}
For $a=5$, \eqref{E:c} gives
\[
f: b \mapsto 
\frac{23}2b-10 +\sqrt{\frac{21\cdot 25}4b^2-250b-625}.
\]
Setting $b_0= 5$, corresponding to the fundamental LEP $a=5,b=5$, the first few terms of the sequence are
$5, 85,1930,44285,1016605$. The first 4 of these are visible in Figure \ref{F:tree35}; they lie on the central horizontal branch to the right of the fundamental solution $(5,5)$.

We use Lemma \ref{L:recu} with $u=\frac{23}2, v=-10, w=-625$.
% -8 16 =2(7+1)(-8)
Thus
\[
b_i=24 b_{i-1} - 24b_{i-2} + b_{i-3},\quad\text{for}\ i\ge 3.
\]
The area $2(a+b)$ is divisible by 5.
Let $A=2(a+b)/5$.
We employ Lemma \ref{L:recu}  with the same value of $u$. Thus
\[
A_i=24 A_{i-1} - 24A_{i-2} + A_{i-3},\quad\text{for}\ i\ge 3.
\]
Setting $A_0= 4$, corresponding to the fundamental LEP $a=5,b=5$, the first few terms of the sequence are
$4,36,774,17716,406644$. 
\end{example}

%%%%%%%%%%%%%%%%%%%%

\section{Diagonals, heights and altitudes}\label{S:dah} 

Let us first fix some terminology and notation; see Figure \ref{F:terms}.

\begin{definition}
Consider a non-square LEP $P$. We denote the length of its long (resp.~short) diagonal $d_l$ (resp.~$d_s$). The \emph{heights} of $P$ are the distances between opposite sides; we denote the long (resp.~short) height $h_l$ (resp.~$h_s$). 
Each diagonal $d$ partitions $P$ into two congruent triangles $T$. We will call the distance from $d$ to the third vertex of $T$ an  \emph{altitude} of $P$. We call the altitude from $d_s$ (resp.~$d_l$) the \emph{long} (resp.~\emph{short}) altitude and denote it $\eta_l$ (resp.~$\eta_s$). 
\end{definition}

\begin{remark} Our notion of altitude is not universal. Some authors use the term altitude for the concept we have called height.
\end{remark}

\begin{figure}[h]
\centering
\begin{tikzpicture}[scale=.7][line cap=round,line join=round,>=triangle 45,x=1cm,y=1cm]
%\clip(-0.5,-4) rectangle (31.259992429996146,9);
\draw[line width=0.8pt,color=qqwuqq,fill=qqwuqq,fill opacity=0.1] (0,4) -- (.27,4) -- (.27,3.73) -- (0,3.73) -- cycle; 
\draw[line width=0.8pt,color=qqwuqq,fill=qqwuqq,fill opacity=0.1] (156/25-1,208/25-4/3) -- (156/25-.81/5-1,208/25-1.08/5-4/3) -- (156/25-.81/5+.216-1,208/25-1.08/5-.162-4/3) -- (156/25+.216-1,208/25-.162-4/3) -- cycle; 
\draw[line width=0.8pt,color=qqwuqq,fill=qqwuqq,fill opacity=0.1] (39/10,13/10) -- (39/10+.256,13/10+.085) -- (39/10+.256-.085,13/10+.085+.256) -- (39/10-.085,13/10+.256) -- cycle; 
\draw[line width=0.8pt,color=qqwuqq,fill=qqwuqq,fill opacity=0.1] (9.23077,-.153846) -- (9.23077+0.224654,-.153846-0.149769) -- (9.23077+0.224654+0.162,-.153846-0.149769+0.216) -- (9.23077+0.162,-.153846+0.216) -- cycle; 
\draw [line width=1.2pt,color=qqqqff] (3,4)-- (0,0);
\draw [line width=1.2pt,color=qqqqff] (3,4)-- (12,4);
\draw [line width=1.2pt,color=qqqqff] (12,4)-- (9,0);
\draw [line width=1.2pt,color=qqqqff] (9,0)-- (0,0);
\draw [line width=0.8pt,dash pattern=on 1pt off 1pt] (0,0)-- (0,4);
\draw [line width=0.8pt,dash pattern=on 1pt off 1pt] (0,4)-- (3,4);
\draw [line width=0.8pt,dash pattern=on 1pt off 1pt] (3,4)-- (156/25-1,208/25-4/3);
\draw [line width=0.8pt,dash pattern=on 1pt off 1pt] (156/25-1,208/25-4/3)-- (11,4-4/3);
\draw [line width=0.8pt,dash pattern=on 1pt off 1pt] (9.23077,-.153846)-- (12,4);
\draw [line width=0.8pt,dash pattern=on 1pt off 1pt] (9.69231,-.461538)-- (9,0);
\draw [line width=0.8pt,dashed] (0,0)-- (12,4);
\draw [line width=0.8pt,dashed] (3,4)-- (9,0);
\draw [line width=0.8pt,dash pattern=on 1pt off 1pt] (39/10,13/10)-- (3,4);
\draw (1.15,2) node {$a$};
\draw (7.5,4.25) node {$b$};
\draw (8,5.36) node {$h_l$};
\draw (-.4,2) node {$h_s$};
\draw (8,3.1) node {$d_l$};
\draw (8,1.1) node {$d_s$};
\draw (10.5,1.15) node {$\eta_l$};
\draw (3.2,2.5) node {$\eta_s$};
\end{tikzpicture}
\caption{Diagonals, heights and altitudes}\label{F:terms}
\end{figure}
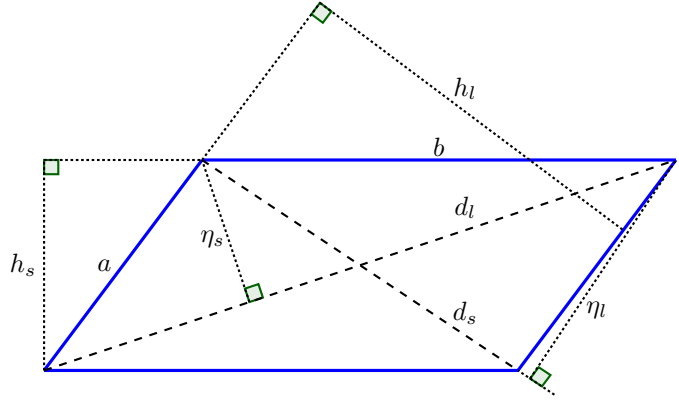

\begin{lemma}\label{L:widths}
Suppose a LEP $P$ has  sides $a,b$ with $a\le b$. Then 
\begin{enumerate}
\item[\rm(a\rm)] $\displaystyle h_l=\frac{2(a+b)}a,\qquad h_s=\frac{2(a+b)}b$.
\item[\rm(b\rm)] $\displaystyle \eta_l=\frac{2(a+b)}{d_s},\qquad \eta_s=\frac{2(a+b)}{d_l}$.
\end{enumerate}
\end{lemma}

\begin{proof}
By equability, the area of $P$ is $2(a+b)$, but the area is obviously also $ah_l$ and $bh_s$. This gives (a).
But the area of $P$ is also twice the area of the triangle determined by each diagonal. So the area of $P$ is both $\eta_l d_s$ and $\eta_s d_l$.  This gives~(b).
\end{proof}

\begin{remark}\label{R:widths}
If  $h_s$  is an integer, then $h_s=2+\frac{2a}b$ by Lemma~\ref{L:widths}(a), and so either $b=a$ or $b=2a$. These cases were treated in Corollary \ref{C:rhom}. 
Note also that the above lemma also gives $(h_s-2)(h_l-2)=4$.
So $h_s,h_l$ are both integers only in the cases $h_s=h_l=4$ and $h_s=3, h_l=6$. The first is the $4\times 4$ square. The second case is the $3\times 6$ rectangle.
 \end{remark}

Equable parallelograms tend to be very thin in form. The following result gives a sharp statement of this thinness.

\begin{theorem}\label{T:altbd} For every LEP the altitudes satisfy  $2<\eta_l\le 2\sqrt5$ and $2<\eta_s\le 2\sqrt2$.
\end{theorem}

\begin{proof}
Suppose a LEP $P$ has  sides $a,b$ with $a\le b$. 
With the notation used in Section \ref{S:345}, we have
$ab= k(m^2+n^2), a+b=kmn, \sqrt{a^2 b^2 -4(a+b)^2}= k(n^2-m^2)$,
for some positive integer $k$ and relatively prime integers $m,n$ with 
$m\le n$. 
From Lemma \ref{L:diag}, the diagonals are given by
\begin{align*}
d_l^2&=(a^2  +  b^2) +\sqrt{4a^2 b^2 -16(a+b)^2}\\
&=k^2m^2n^2-2k(m^2+n^2)+2k(n^2-m^2)=m^2(k^2n^2-4k),\\
d_s^2&=(a^2  +  b^2)- \sqrt{4a^2 b^2 -16(a+b)^2}\\
&=k^2m^2n^2-2k(m^2+n^2)-2k(n^2-m^2)=n^2(k^2m^2-4k).
\end{align*}
The long altitude is
\[
\eta_l=\frac{2(a+b)}{d_s}=\frac{2knm}{n\sqrt{k^2m^2-4k}}.
\]
Squaring and rearranging gives $ km^2 = \frac{4\eta_l^2}{\eta_l^2-4}$. As $m\ge 1$ and $k\ge 5$ by the proof of Theorem~\ref{T:345},
we have $\frac{4\eta_l^2}{\eta_l^2-4}\ge 5$ and so $20\ge \eta_l^2$, that is $\eta_l\le 2\sqrt5$.
Furthermore, 
\[
\eta_l=\frac{2knm}{n\sqrt{k^2m^2-4k}}>\frac{2knm}{n\sqrt{k^2m^2}}=2.
\]

 Arguing in the same manner one finds $\eta_s>2$ and $ kn^2 = \frac{4\eta_s^2}{\eta_s^2-4}$. 
As we saw in Section \ref{S:forest}, for $k=9$ (resp.~8, resp.~5), the fundamental solution has $(m,n)=(1,1)$ (resp.~(1,1), resp.~(1,2)). 
So the minimum value of $kn^2$ is 8. Rearranging $ 8 \le  \frac{4\eta_s^2}{\eta_s^2-4}$ gives $\eta_s^2 \le 8$, that is $\eta_s\le 2\sqrt2$.\end{proof}

\begin{remark}
The bound $\eta_l= 2\sqrt5$ is attained by the rhombus of Figure \ref{F:rhom}, and 
the bound $\eta_s= 2\sqrt2$ is attained by the $4\times 4$ square.
\end{remark}

We complete this section with some results that use either the above theorem, or ideas from its proof.

\begin{proposition}\label{P:irrat}
For every LEP, the two diagonals and the two altitudes are irrational.
\end{proposition}

\begin{proof}
As we saw in the above proof, $d_l^2=m^2(k^2n^2-4k)$. So the longer diagonal has integer length only when  $k^2n^2-4k$ is a square. But $k=5, 8$ or $9$, 
as we saw in the proof of Theorem~\ref{T:345}. For $k=5$ (resp.~8), the expression $k^2n^2-4k$ is divisible by $5$ (resp.~$4k=2^5$) but not by $5^2$ (resp.~$2^6$) and is hence never a square. For $k=9$, one has $k^2n^2-4k=3^2(9n^2-4)$, and $9n^2-4$ is not a square for any integer $n$. Thus the long diagonal is never an integer, and hence  by Remark~\ref{L:int}, the long diagonal is irrational for every LEP. Similarly, the short diagonal is always irrational as  $k^2m^2-4k$ is never a square.
Consequently, by Lemma \ref{L:widths}(b), the altitudes  $\eta_l,\eta_s$ are also  irrational for every LEP.
\end{proof}

The following immediate corollaries of Proposition \ref{P:irrat} are each generalisations of 
Proposition \ref{P:partit}. Recall that a triangle is \emph{Heronian} if it has integer side lengths and integer area.

\begin{corollary} No LEP can be partitioned into the union of two Heronian triangles. 
\end{corollary}

\begin{corollary} There exists no LEP having a pair of opposite vertices with the same $x$-coordinate, or the same $y$-coordinate. 
\end{corollary}

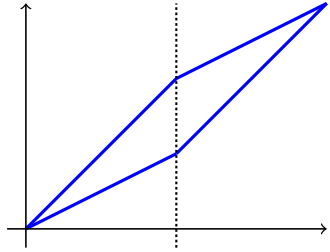
\begin{figure}[h]
\centering
\begin{tikzpicture}[scale=.5][line cap=round,line join=round,>=triangle 45,x=1cm,y=1cm]
%\clip(-0.5,-4) rectangle (31.259992429996146,9);
\draw [line width=1.2pt,color=qqqqff] (4,2)-- (0,0);
\draw [line width=1.2pt,color=qqqqff] (4,2)-- (8,6);
\draw [line width=1.2pt,color=qqqqff] (0,0)-- (4,4);
\draw [line width=1.2pt,color=qqqqff] (4,4)-- (8,6);
\draw [line width=0.8pt,dash pattern=on 1pt off 1pt] (4,-.5)-- (4,6);
\draw [->,line width=0.6pt] (-.5,0) -- (8,0);
\draw [->,line width=0.6pt] (0,-.5,0) -- (0,6);
\end{tikzpicture}
\caption{A parallelogram that is not a LEP}\label{F:notlep}
\end{figure}

\begin{proposition}\label{P:ovaxis} If a LEP $P$ contains the origin $O$ and vertices $O, A, B, C$ in cyclic order with long diagonal $OB=d_l$ belonging to the 1st quadrant, then all four vertices belong to the first quadrant.
\end{proposition}

\begin{figure}[h]
\begin{tikzpicture}[line cap=round,line join=round,>=triangle 45,x=1.0cm,y=1.0cm]
\draw [line width=.1pt,dash pattern=on 3pt off 3pt, xstep=1.0cm,ystep=1.0cm] (-0.36,-1.24) grid (7.66,3.54);
\draw[->,color=black] (-0.36,0) -- (7.66,0);
\foreach \x in {,1,2,3,4,5,6,7}
\draw[shift={(\x,0)},color=black] (0pt,-2pt);
\draw[->,color=black] (0,-1.24) -- (0,3.54);
\foreach \y in {-1,1,2,3}
\draw[shift={(0,\y)},color=black] (2pt,0pt) -- (-2pt,0pt);
\draw [line width=1.2pt,color=qqqqff] (0,0)-- (3,-1);
\draw [line width=1.2pt,color=qqqqff] (3,-1)-- (7,2);
\draw [line width=1.2pt,color=qqqqff] (7,2)-- (4,3);
\draw [line width=1.2pt,color=qqqqff] (4,3)-- (0,0);
\draw [dash pattern=on 3pt off 3pt] (0,0)-- (7,2);
\draw [dash pattern=on 3pt off 3pt] (3,-1)-- (2.51,0.715);
\draw (4.38,1.5) node {$d_l$};
\draw (2.46,0.25) node {$\eta_s$};
\fill [color=black] (0,0) circle (1.5pt);
\draw[color=black] (-0.15,0.18) node {$O$};
\fill [color=black] (3,-1) circle (1.5pt);
\draw[color=black] (3.92,-1.02) node {$A(x,y)$};
\fill [color=black] (7,2) circle (1.5pt);
\draw[color=black] (7.2,2.2) node {$B$};
\fill [color=black] (4,3) circle (1.5pt);
\draw[color=black] (4.3,3.16) node {$C$};
\end{tikzpicture}
\caption{A LEP configuration that cannot exist}\label{F:1stq}
\end{figure}
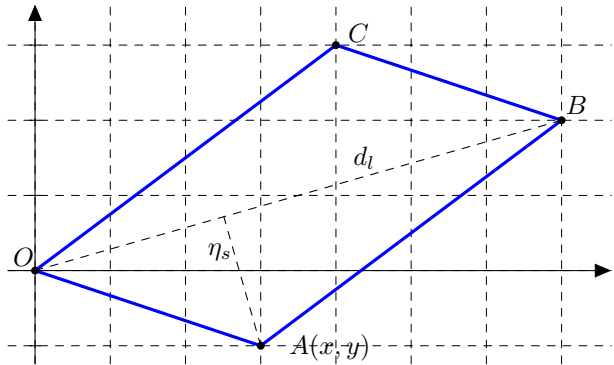

\begin{proof}
By applying if necessary a symmetry along the $x=y$  axis we suppose that $A$ lies in the 4th quadrant as in Figure \ref{F:1stq}. Since the diagonals are irrational the vertex $B$ cannot be on the $x$-axis. Moreover, as $P$ has an acute angle at $O$, the foot of the altitude from $A$ to $d_l$ also lies in the first quadrant. Hence, letting $A=(x,y)$, the distance from $A$ to $d_l$ is greater than $|y|$. Thus, by 
Theorem \ref{T:altbd}, $|y|< \eta_s\le  2\sqrt2$. Hence, as $y$ is an integer, $|y|$ is 1 or 2. Let $OA$ have length $a$. Then $a^2$ is $x^2+1$ or $x^2+2$. But this is impossible for $x>0$. 
\end{proof}

\begin{corollary} Every LEP can be moved by Euclidean motion to a LEP in the 1st quadrant with a vertex at the origin. 
\end{corollary}

%%%%%%%%%%%%%%%%%%%%

\section{Pythagorean Equable Parallelograms}\label{S:PEPs} 

In general there are two ways in which parallelograms can be circumscribed by a rectangle so that the rectangle and parallelogram share a common diagonal; the rectangle may have sides extending the long sides of the parallelogram, or the short sides of the parallelogram; see Figure \ref{F:circums}.

\begin{figure}[H]
\centering
\begin{tikzpicture}[scale=.4][line cap=round,line join=round,>=triangle 45,x=1cm,y=1cm]
%\clip(-0.5,-4) rectangle (31.259992429996146,9);
\draw [line width=1.2pt,color=qqqqff] (3,4)-- (0,0);
\draw [line width=1.2pt,color=qqqqff] (3,4)-- (11,4);
\draw [line width=1.2pt,color=qqqqff] (11,4)-- (8,0);
\draw [line width=1.2pt,color=qqqqff] (8,0)-- (0,0);
\draw [line width=0.8pt,dash pattern=on 1pt off 1pt] (0,0)-- (0,4);
\draw [line width=0.8pt,dash pattern=on 1pt off 1pt] (0,4)-- (3,4);
\draw [line width=0.8pt,dash pattern=on 1pt off 1pt] (8,0)-- (11,0);
\draw [line width=0.8pt,dash pattern=on 1pt off 1pt] (11,0)-- (11,4);
\draw [line width=0.8pt,dash pattern=on 1pt off 1pt] (3,4)-- (5.88,7.84);
\draw [line width=0.8pt,dash pattern=on 1pt off 1pt] (5.88,7.84)-- (11,4);
\draw [line width=0.8pt,dash pattern=on 1pt off 1pt] (0,0)-- (5.12,-3.84);
\draw [line width=0.8pt,dash pattern=on 1pt off 1pt] (5.12,-3.84)-- (8,0);
\draw [line width=0.8pt,dashed] (0,0)-- (11,4);
\end{tikzpicture}
\caption{Rectangles circumscribing a parallelogram}\label{F:circums}
\end{figure}
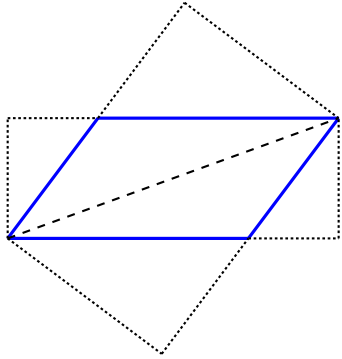

\begin{definition}\label{D:PEPs}
A LEP is said to be \emph{Pythagorean} if it is circumscribed by a rectangle having integer side lengths so that the rectangle and parallelogram share a common diagonal as in Figure \ref{F:circums}.
\end{definition}

This terminology is justified by the equivalent condition (b) in the following result.

\begin{proposition}\label{P:pythag} Consider a LEP $P$ having sides $a,b$ with $a\le b$. The following conditions are equivalent:
\begin{enumerate}
\item[\rm(a\rm)] $P$ is Pythagorean,
\item[\rm(b\rm)] $P$ is  circumscribed by a rectangle $R$ such that the complement of $P$ in $R$ is the union of two Pythagorean triangles,
\item[\rm(c\rm)] $a$ divides $2b$,
\item[\rm(d\rm)] $P$ can be drawn as a LEP with a horizontal pairs of sides.
\end{enumerate}
\end{proposition}

\begin{figure}[H]
\begin{tikzpicture}[scale=.38][line cap=round,line join=round,>=triangle 45,x=1cm,y=1cm]
%\clip(-0.5,-1) rectangle (31.259992429996146,9);
\draw [line width=1.2pt,color=qqqqff] (18,0)-- (22,0);
\draw [line width=1.2pt,color=qqqqff] (22,0)-- (30,8);
\draw [line width=1.2pt,color=qqqqff] (30,8)-- (26,8);
\draw [line width=1.2pt,color=qqqqff] (26,8)-- (18,0);
\draw [line width=0.8pt,dash pattern=on 1pt off 1pt] (22,0)-- (30,0);
\draw [line width=0.8pt,dash pattern=on 1pt off 1pt] (30,0)-- (30,8);
\draw [line width=0.8pt,dash pattern=on 1pt off 1pt] (18,0)-- (18,8);
\draw [line width=0.8pt,dash pattern=on 1pt off 1pt] (18,8)-- (26,8);
\draw [->,line width=0.8pt] (17.2,5) -- (17.2,8);
\draw [->,line width=0.8pt] (17.2,2.9) -- (17.2,0);
\draw [->,line width=0.8pt] (22,8.6) -- (18,8.6);
\draw [->,line width=0.8pt] (23,8.6) -- (26,8.6);
\draw (17.2,4) node {$h_l$};
\draw (22.5,8.6) node {$x$};
%\draw [->,line width=0.8pt] (19.5,-0.5) -- (18,-0.5);
%\draw [->,line width=0.8pt] (20.5,-0.5) -- (22,-0.5);
%\draw [->,line width=0.8pt] (27.015187039839827,3.708485338939847) -- (22.811318060847118,-0.5833306898225437);
%\draw [->,line width=0.8pt] (28.10573045698438,4.7638499361765) -- (30.83208899984576,7.6133343487154646);
\draw (26.7,4) node {$b$};
\draw (20,-.6) node {$a$};
\draw [line width=0.8pt,dashed] (18,0)-- (30,8);
\end{tikzpicture}
\caption{Integer $h_l$ $\implies$ integer $x$ }\label{F:hl}
\end{figure}
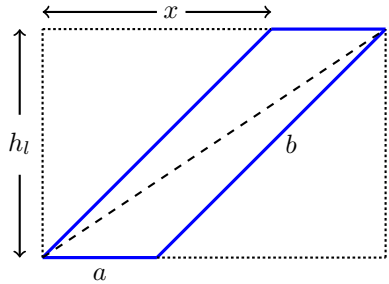

\begin{proof}[Proof of Proposition \ref{P:pythag}]
Suppose that $P$ is a LEP with vertices 
$O, A,B,C$,
in cyclic order, and that $P$ has  sides $a,b$ with $a\le b$.

(a) $\implies$ (d). Suppose that $P$ is Pythagorean and that the distance between sides $OA$ and $BC$ is an integer, $h$ say. Rotate and translate $P$ so that it has the side $OA$ along the positive $x$-axis, with the vertex  $O$ at the origin. We want to verify that moved in this way, $P$ is still a LEP. Since $P$ has integer side lengths, $A$ has integer coordinates. By assumption, $B,C$ have integer $y$-coordinate, equal to $h$. From the Pythagorean hypothesis, $B$ has integer $x$-coordinate, $z$ say. Then $x$-coordinate of $C$ is $z$ minus the $x$-coordinate of $A$. So $P$ is a LEP.

(d) $\implies$ (c). If a LEP $P$ has a horizontal pairs of sides, then the fact that the coordinates of the vertices are integers implies that the distance $h$ between the horizontal sides is an integer. So either $h_l$ or $h_s$ is an integer.  If $h_s$ is an integer, then as $h_s=2+\frac{2a}b$ by Lemma~\ref{L:widths}(a), and $a\le b$, we have either $b=a$ or $b=2a$. In either case, by Lemma~\ref{L:widths}(a) again, $h_l=2+\frac{2b}a$ is also an integer. Thus $a$ divides $2b$.

(c) $\implies$ (b). Since $h_l=2+\frac{2b}a$, hence $h_l$ is an integer. Consider the rectangle $R$ that circumscribes $P$, sharing a common diagonal with $P$ and having sides that are extensions of the sides of $P$ of length $a$, as in Figure \ref{F:hl} (but not necessarily with the sides $a$ horizontal). So $R$ has sides $h_l$ and $ a+x$, for some $x$, and we are required to show that $x$ is an integer. By Lemma~\ref{L:widths}(a), we have
\[
x^2=b^2-h_l^2=b^2-\frac{4(a+b)^2}{a^2}=\frac{a^2b^2-4(a+b)^2}{a^2},
\]
so $x=\frac{\sqrt{a^2b^2-4(a+b)^2}}a$, which is a rational by Theorem \ref{T:suff}. Since $h_l=2+\frac{2b}a$ is  an integer, $2b=ai$ for some integer $i$. Then 
\[
x^2=\frac{a^2b^2-4(a+b)^2}{a^2}=b^2-(2+i)^2,
\]
 which is an integer. So $x$ is rational and the square root of an integer. Hence $x$ is an integer.

(b) $\implies$ (a). This is immediate.
\end{proof}

\begin{remark}
The smallest non-Pythagorean LEP has area 180  and sides 25 and 65; it can be constructed in the first quadrant with vertices $(0,0),(25,60),(32,84),(7,24)$.
\end{remark}

\begin{proof}[Proof of Theorem \ref{T:PEPs}]
Consider a Pythagorean LEP $P$ with sides $a,b$. By Proposition \ref{P:pythag}, $a$ divides $2b$, and by Theorem \ref{T:345}, $\gcd(a,b)=3,4$ or $5$. If $a$ is odd, then $a$ divides $b$, so $a=\gcd(a,b)$ and $a$ is 3 or 5. If $a$ is even, say $a=2a'$, then $a'$ divides $b$ so $a'$ divides $\gcd(a,b)$. Thus $a'$ is $2,3,4$ or $5$ and thus $a$ is $4,6,8$ or $10$.
But as observed in Remark \ref{R:div}, if $\gcd(a,b)=4$  neither $a$ nor $b$ is divisible by $8$.  So $a$ is 3,4,5,6 or 10, as claimed in Theorem \ref{T:PEPs}. In remains to exhibit a family of such  Pythagorean LEPs in each case.

In the notation of previous sections, $ab= k(m^2+n^2), a+b=kmn$
for  positive integers $k,m,n$  
where $m,n$ are relatively prime with $m<n$. For convenience, let us restate \eqref{E:xb}:
\[
km^2+kn^2+a^2=kamn.
\]
Moreover, as we saw in the proof of Theorem~\ref{T:345}, $k$ is either 5, 8 or 9, corresponding to $\gcd(a,b)=3,4,5$ respectively.

Note that for $a=3$ we have $k=9$  and \eqref{E:xb} gives $m^2+n^2+1=3mn$. Setting $x=n-m,y=m+n$ we obtain the Pell-like equation $y^2-5x^2 =4$. Conversely, note that if $(x,y)$ is a solution to this equation, then $x,y$ necessarily have the same parity and we can set $m=\frac{y-x}2,n=\frac{y+x}2 $ to obtain an integer solution $(m,n)$ to $m^2+n^2+1=3mn$. Then  $b=\frac{k}a (m^2+n^2)=\frac{3(x^2+y^2)}{2}$.
\begin{comment}
$=\frac{3(6x^2+4)}{2}=6+9x^2$.
\end{comment}

For $a=4$ we have $k=8$  and \eqref{E:xb} gives $m^2+n^2+2=4mn$. Note then that $m$, $n$ must have same parity. Setting $2x=n-m$, $2y=n+m$ we obtain the Pell equation $y^2 -3x^2=1$.
Conversely, note that if $(x,y)$ is a solution to this equation, then  we can set $m=y-x,n=y+x $ to obtain an integer solution $(m,n)$ to $m^2+n^2+2=4mn$. 
Then  $b=\frac{k}a (m^2+n^2)=4(x^2+y^2)$.
\begin{comment}
$=4(4x^2+1)=4+16x^2$.
\end{comment}

For $a=5$ we have $k=5$  and \eqref{E:xb} gives $m^2+n^2+5=5mn$. Setting $x=n-m,y=m+n$ we obtain the Pell-like equation $3y^2-7x^2 =20$. Conversely, note that if $(x,y)$ is a solution to this equation, then $x,y$ necessarily have the same parity and we can set $m=\frac{y-x}2,n=\frac{y+x}2 $ to obtain an integer solution $(m,n)$ to $m^2+n^2+5=5mn$. 
Then  $b=\frac{k}a (m^2+n^2)=\frac{x^2+y^2}2$.

Similarly, for $a=6$ (resp.~$a=10)$, we have $k=9$ (resp.~$k=5$) and \eqref{E:xb} gives $m^2+n^2+4=6mn$ (resp.~$m^2+n^2+20=10mn$). Note that $m,n$ necessarily have the same parity and we can set $x=\frac{n-m}2,y=\frac{m+n}2 $. This gives the Pell (resp.~Pell-like) equation $y^2-2x^2 =1$ (resp.~$2y^2-3x^2 =5$). Conversely, note that if $(x,y)$ is a solution to this equation, then we can set $m=y-x,n=y+x $ to obtain an integer solution $(m,n)$ to $m^2+n^2+4=6mn$ (resp.~$m^2+n^2+20=10mn$).
Using $b=\frac{k}a (m^2+n^2)$, for $a=6$ we have $b=3(x^2+y^2)$ and for $a=10$ we have $b=x^2+y^2$.
\end{proof}

%%%%%%%%%%%%%%%%%%%%%%%%%

We complete this section by connecting the above solution families F1-F5 to the material in the earlier Sections. In particular, we locate the Pythagorean equable parallelograms in the trees of LEPs of Figure \ref{F:tree35} and  \ref{F:tree4}.

{\bf F1: \boldmath$a=3$}. The solutions $(x,y)$ to the Pell-like equation $y^2 - 5x^2=4$ are well known to be $(L_{2i},F_{2i})$, starting at $(F_0,L_0)=(0,2),(F_2,L_2)=(1,3)$, where $F$ stands for the Fibonacci numbers and $L$ for the Lucas numbers. Both $L_{2i}$ and $F_{2i}$ satisfy the recurrence relation  $z_{i}=3z_{i-1}-z_{i-2}$, with different initial conditions. So the first few solutions of  equation $y^2-5x^2 =4$ are: $(0,2), (1,3), (3,7), (8,18)$, $(21,47)$. The resulting $b$-values, given by $b=\frac{3(x^2+y^2)}{2}$, are: 6, 15, 87, 582, 3975.
We saw the corresponding LEPs $(3,b)$ previously in Example \ref{Eg:3}; in Figure \ref{F:tree35} they lie on the central horizontal branch to the right of the fundamental solution $(3,6)$. As we saw,  these $b$-values are generated recursively by the formula $f(b)=\frac{7}2b-6+\sqrt{\frac{45}4 b^2-54b-81}$, or alternatively, using the recurrence relation  $b_i=8 b_{i-1} - 8b_{i-2} + b_{i-3}$.
Notice that by Theorem \ref{T:suff}, these $b$ values are precisely those numbers $b=3i$ for which $5i^2-8i-4$ is a square.

{\bf F2: \boldmath$a=4$}. The solutions $(x,y)$ to  Pell's equation $y^2 -3x^2=1$ are well known to be given by the recurrence relation  $z_i = 4z_{i-1} - z_{i-2}$. See OEIS entries   A001075 and A001353 \cite{OEIS}. The first few solutions of $y^2 -3x^2=1$ are: $(0,1),(1,2),(4,7),(15,26),(56,97)$. The resulting $b$-values, given by $b=4(x^2+y^2)$, are: 4, 20, 260, 3604, 50180. We saw the corresponding LEPs $(4,b)$ in Example \ref{Eg:4}; in Figure \ref{F:tree4} they lie on the central horizontal branch to the right of the fundamental solution $(4,4)$.
The $b$-values can be generated recursively by the formula $f(b)=7b-8+4\sqrt{3b^2-8b-16}$, or by  the recurrence relation  $b_i=15 b_{i-1} - 15b_{i-2} + b_{i-3}$.
Notice that by Theorem \ref{T:suff}, these $b$ values are precisely those numbers $b=4i$ for which $3i^2-2i-1$ is a square.

{\bf F3: \boldmath$a=5$}. The Pell-like equation $3y^2-7x^2 =20$ is less common. Its solutions $(x,y)$ are given by the recurrence relation  $z_i = 5z_{i-2} - z_{i-4}$.
The first few solutions are: $(1,3),(2,4),(7,11),(11,17),(34,52),(53,81),(163,249)$. The resulting $b$-values, given by $b=\frac{x^2+y^2}2$, are: 5, 10, 85, 205, 1930, 4685, 44285.
The position of the corresponding LEPs $(5,b)$ is more complicated than what we saw for $a=3$ and $a=4$. Every second solution, starting at $(5,5)$, appeared  in Example \ref{Eg:5}; in Figure \ref{F:tree35} they lie on the central horizontal branch to the right of the fundamental solution $(5,5)$. These $b$-values can be generated recursively by the formula 
\[
f(b)=\frac{23}2b-10 +\sqrt{\frac{21\cdot 25}4b^2-250b-625},
\]
or by  the recurrence relation  $b_i=24 b_{i-1} - 24b_{i-2} + b_{i-3}$. The other solutions are on the central horizontal branch of Figure \ref{F:tree35} to the left of the fundamental solution $(5,5)$. These $b$-values are generated recursively by the same formula (and recurrence relation) but starting at $b=10$. 

Note that in the proof of Theorem \ref{T:PEPs} the equation $3y^2-7x^2 =20$ was derived from  $m^2+n^2+5=5mn$. This latter equation is well known; indeed, it is easy to see that for the solutions $m,n$, the first component comprise those numbers $m$ for which $21m^2-20$ is a square. See entry A237254 in \cite{OEIS}.
By Theorem \ref{T:suff}, the $b$ values for $a=5$ are precisely those numbers $b=5i$ for which $21i^2-8i-4$ is a square.

\begin{comment}
If $m^2+n^2+5=5mn$, then setting $x=n,y=-2m+5n$ we obtain the Pell-like equation $y^2-21x^2 =-20$. Conversely, note that if $(x,y)$ is a solution to this equation, then $x,y$ necessarily have the same parity and we can set $m=\frac{5x-y}2,n=x $ to obtain an integer solution $(m,n)$ to $m^2+n^2+5=5mn$. 
\end{comment}

{\bf F4: \boldmath$a=6$}. The  solutions $(x,y)$ to  Pell's equation $y^2 -2x^2=1$ are well known to be given by the recurrence relation   $z_i = 6z_{i-1} - z_{i-2}$. 
See OEIS entries  A001541, A001542 \cite{OEIS}. The first few solutions are: $(0,1),(2,3),(12,17),(70,99),(408,577)$. The resulting $b$-values, given by $b=3(x^2+y^2)$, are: 3, 39, 1299, 44103, 1498179.
The corresponding LEPs $(6,b)$  occur on the central horizontal branch of Figure \ref{F:tree35}, to the left of the fundamental solution $(3,6)$.
To generate the $b$-values of the solutions on this branch we can use Equation \ref{E:c} with $a=6$. These $b$-values can be generated recursively by the formula 
\[f(b)
= 17b-12+6\sqrt{9b^2-(6+b)^2},
\]
and by the recurrence relation  $b_i=35 b_{i-1} - 35b_{i-2} + b_{i-3}$.
By Theorem \ref{T:suff}, these $b$ values are precisely those numbers $b=3i$ for which $2i^2-i-1$ is a square.

{\bf F5: \boldmath$a=10$}. Like the $a=5$ case, the Pell-like equation $2y^2-3x^2 =5$ is not very common. Its solutions $(x,y)$ are given by the recurrence relation  $z_i = 10z_{i-2} - z_{i-4}$.
The first few solutions are: $(1,2),(3,4),(13,16),(31,38),(129,158),(307,376)$. The resulting $b$-values, given by $b=x^2+y^2$, are: 5, 25, 425, 2405, 41605, 235625.
The corresponding LEPs $(10,b)$ occur on the first vertical  branch  to the left of the fundamental solution $(5,5)$ in Figure \ref{F:tree35}. Every second solution, starting at $(10,5)$ appears above the central horizontal axis. To generate these $b$-values we can use Equation \ref{E:c} with $a=10$. This gives the equation
\[f(b)
=  49b-20+10\sqrt{25b^2-(10+b)^2},
\]
and by \eqref{E:crel}, the recurrence relation  $b_i=99 b_{i-1} - 99b_{i-2} + b_{i-3}$.
The other solutions appears below the central horizontal axis. Their $b$-values  are generated recursively by the same formula (and recurrence relation) but starting at $b=25$.

Note that in the proof of Theorem \ref{T:PEPs} the equation $2y^2-3x^2 =5$ was derived from  $m^2+n^2+20=10mn$. This latter equation is well known; indeed, it is easy to see that for the solutions $m,n$, the first component comprise those numbers $m$ for which $6m^2-5$ is a square. See entry  A080806 in \cite{OEIS}.
By Theorem \ref{T:suff}, the $b$ values for $a=10$ are precisely those numbers $b=5i$ for which $6i^2-i-1$ is a square.

\begin{comment}
If $m^2+n^2+20=10mn$, then $m,n$ necessarily have the same parity and we can set $x=-10m+n,y=\frac{-49m+5n}2 $. This gives the Pell-like equation $y^2-6x^2 =-5$. Conversely, if $(x,y)$ is a solution to this equation, then we can set $m=5x-2y,n=49x-20y $ to obtain an integer solution $(m,n)$ to $m^2+n^2+20=10mn$.
\end{comment}

\begin{comment}
Recall that by Proposition \ref{P:pythag}, a LEP with sides $a,b$ with $a\le b$ is Pythagorean precisely when $a$ divides $2b$. Note that in  F1-F5 above, the relevant $b$-value in Figures \ref{F:tree35} \ref{F:tree4}, as determined recursively as described above,  all give Pythagorean equable parallelograms, because the respective  third order recurrence relations clearly preserve the condition that $a$ divides $2c$. 
\end{comment}

\begin{remark}
The Pell-like equations in families F1 -- F5 can all be put in the following form:
\begin{equation}\label{E:eq}
cy^2-(c+4)x^2=d,
\end{equation}
 where $c,d$ are integers and $d$ is even if $c$ is odd. Table~\ref{table} gives the values of $c,d$ for the five families.

A  direct calculation shows that if $(x,y)$ is a solution to \eqref{E:eq}, then another solution is given by
\[
\begin{pmatrix}
x'\\
y'
\end{pmatrix}=\frac12\begin{pmatrix}
c+2&c\\
c+4&c+2
\end{pmatrix}\begin{pmatrix}
x\\
y
\end{pmatrix}=\begin{pmatrix}
x+\frac{c(x+y)}2\\
2x+y+\frac{c(x+y)}2
\end{pmatrix}.
\]
This was proved by R\'ealis for $d=\pm 4$  \cite{Re} (see also \cite[p.~407]{Di}), but it holds in the general case. (Note that if $c$ is odd, then from  \eqref{E:eq}, as $d$ is even, the  numbers $x,y$ must have the same parity, so $x',y'$ are integers).  Furthermore, provided \eqref{E:eq} has one solution, $(x_0,y_0)$ say, this process gives a sequence of solutions $(x_{i},y_{i})$, and it is easy to verify that it satisfies the following recurrence relation:
\[
(x_{i},y_{i})= (c+2) (x_{i-1},y_{i-1})-(x_{i-2},y_{i-2}).
\]
Note that \eqref{E:eq} does not  have a solution for all $c$ and $d$. For example, for $c=3, d=4$, the equation is $3y^2-7x^2=4$. Modulo 3 this is $-x^2\equiv 1$, which has no solution.
 
\begin{table}[h]
\begin{center}
\begin{tabular}{c|ccccc}
  \hline
  Family & F1 & F2  & F3& F4& F5 \\
  $c$ & 1 & 2  & 3& 4& 8 \\
  $d$ & 4 & 2  & 20& 4& 20 \\
  \hline
\end{tabular}
\end{center}
\caption{\ }
\label{table}
\end{table}

\end{remark}

%%%%%%%%%%%%%%%%%%%%%%%%%

\section*{Appendix: The Equable Triangles Theorem} 

The Equable Triangles Theorem says that there are only 5 equable triangles with integer side lengths. For the early history of the theorem, Dickson \cite[pp.~195, 199]{Di} cites Whitworth and Briddle in Math. Quest. Educational Times 5, 1904, 54--56, 62--63, but we have been unable to locate this/these works. Equable triangles have been investigated in several works \cite{Kil,Ma,Sm,Wie}. Proofs of the Equable Triangles Theorem are given in  \cite{Fo} and \cite[pp.15--16]{Br}. In this appendix, the proof we supply follows the argument in \cite{Br}, but pushes it a little further so that the conclusion can be readily done without resort to a computer.

For equable triangles with sides $a,b,c$, Heron's formula gives
\[
(a+b+c)(-a+b+c)(a-b+c)(a+b-c)=16(a+b+c)^2.
\]
Let $u=-a+b+c, v=a-b+c,w=a+b-c$, so that $a=\frac{v+w}2,b=\frac{u+w}2,c=\frac{u+v}2$.
Then our equation is 
\begin{equation}\label{E:1}
uvw=16(u+v+w),
\end{equation}
and we look for solutions $u,v,w$, all of the same parity. Further, we may assume $u\le v\le w$. 
Note that, since $u,v,w$, have the same parity, so  from \eqref{E:1}, $u,v,w$ are necessarily even. Let $u=2x,v=2y,w=2z$, so $a=y+z,b=x+z,c=x+y$. Then
$xyz=4(x+y+z)$.
Thus
\[
y\le z=\frac{4(x+y)}{xy-4},
\]
so
$xy^2-8y-4x\le 0$.
Hence 
\[
x\le y\le \frac{4+2\sqrt{4+x^2}}x\]
so
$x^2\le 4+2\sqrt{4+x^2}$.
Hence
$(x^2-4)^2\le 4(4+x^2)$.
Thus
$x^4-12x^2\le 0$,
which gives $x\le 3$. Then 
\[
y\le \frac{4+2\sqrt{4+x^2}}x\le 4+2\sqrt{4+1},
\]
 since the function $\frac{4+2\sqrt{4+x^2}}x$ is decreasing for positive $x$. So, as $y$ is an integer,  $y\le 8$. Then, considering the values $x\le 3$, $y\le 8$ and 
$z=\frac{4(x+y)}{xy-4}$, we find the following integer values for $x,y,z$:
\[
1,5,24\qquad
1,6,14\qquad
1,8,9\qquad
2,3,10\qquad
2,4,6,
\]
which give the values for $a,b,c$:
\[
6,25,29\qquad
7,15,20\qquad
9,10,17\qquad
5,12,13\qquad
6,8,10.
\]

\begin{figure}
\definecolor{qqqqff}{rgb}{0,0,1}
\definecolor{ttzzqq}{rgb}{0.2,0.6,0}
\begin{tikzpicture}[scale=.75][line cap=round,line join=round,>=triangle 45,x=1.0cm,y=1.0cm]
\clip(-0.68,-0.66) rectangle (6.92,4.22);
\draw[color=ttzzqq,fill=ttzzqq,fill opacity=0.1] (6,0.42) -- (5.58,0.42) -- (5.58,0) -- (6,0) -- cycle; 
\fill[color=qqqqff,fill=qqqqff,fill opacity=0.1] (0,0) -- (3,0) -- (6,4) -- cycle;
\draw (0,0)-- (3,0);
%\draw (3,0)-- (6,4);
\draw (3,0) -- (6,4) node [midway,below, sloped] (3,0) {$b\,=\,5k$};
    \draw (6,4)-- (0,0);
\draw [dotted] (3,0)-- (6,0);
\draw [dotted] (6,0)-- (6,4);
\draw (4.32,0.0) node[anchor=north west] {$3k$};
\draw (6.14,2.42) node[anchor=north west] {$4k$};
%\draw (3.44,2.08) node[anchor=north west] {$b\,=\,5k$};
\draw (1.34,0.0) node[anchor=north west] {$a$};
\draw (2.64,2.6) node[anchor=north west] {$c$};
\draw [dotted,color=qqqqff] (0,0)-- (3,0);
\draw [dotted,color=qqqqff] (3,0)-- (6,4);
\draw [dotted,color=qqqqff] (6,4)-- (0,0);
\begin{scriptsize}
\fill [color=black] (0,0) circle (1.5pt);
\draw[color=black] (-0.24,-0.46) node {$B$};
\fill [color=black] (3,0) circle (1.5pt);
\draw[color=black] (3.42,-0.46) node {$C$};
\fill [color=black] (6,4) circle (1.5pt);
\draw[color=black] (6.32,3.76) node {$A$};
\fill [color=black] (6,0) circle (1.5pt);
\draw[color=black] (6.32,0.26) node {$D$};
\end{scriptsize}
\end{tikzpicture}\caption{Equable triangles}\label{F:tri}
\end{figure}
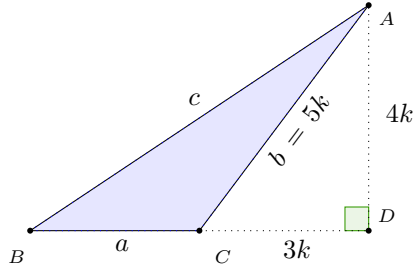

This completes the proof. Note that the last two  are Pythagorean triples, while intriguingly, for the other three cases, the corresponding triangle is the complement of a Pythagorean triangle of the form $3k,4k,5k$ in a larger Pythagorean triangle; see Figure \ref{F:tri}. The three cases correspond to the values $k=2,3,5$.

%%%%%%%%%%%%%%%%%%%%%%%%%

\begin{thks} The authors are very grateful to the referee whose thoughtful suggestions improved the presentation of the paper. The first author would like to warmly thank John Steinig for introducing him to number theory and revealing its intriguing splendour many years ago.
\end{thks}

%%%%%%%%%%%%%%%%%%%%%%%%%%%

\bibliographystyle{amsplain}

\begin{thebibliography}{}



\bibitem{BU}  Arthur Baragar  and Kensaku Umeda, The asymptotic growth of integer solutions to the
              {R}osenberger equations,
 \emph{Bull. Austral. Math. Soc.},
{\bf 69} (2004),
no.~3,
{481--497}.

\bibitem{Br}  Christopher J. Bradley, 
\emph{Challenges in geometry},
 {Oxford University Press, Oxford},
     2005.

\bibitem{BS} Tom C. Brown  and Peter Jau-Shyong Shiue, {Squares of second-order linear recurrence sequences},
 \emph{Fibonacci Quart.},
 {\bf 33} (1995), no.~4,
{352--356}.
 

\bibitem{Di} Leonard Eugene Dickson, \emph{History of the theory of numbers. {V}ol. {II}},
{Chelsea Publishing Co., New York}, 1966.

\bibitem{Fo}  Arthur H. Foss, Integer-sided triangles, \emph{Math. Teacher} {\bf  73} (1980), no. 5,  390--392.

\bibitem{Kil} Jean E. Kilmer, Triangles of equal area and perimeter and inscribed circles, \emph{Math. Teacher} {\bf  81} (1981), no.~1,  65--70.

\bibitem{Ma} Lee Markowitz, Area=perimeter, \emph{Math. Teacher} {\bf  74} (1981), no.~3,  222--223.

\bibitem{JO} Jacques Ozanam, \emph{R\'ecr\'eations Math\'ematiques et Physiques}, Tome premier, chez George Gallet, (1698), 110--113.

\bibitem{Re} S. R\'ealis, R\'esolution d'une \'equation ind\'etermin\'ee par formules directes, \emph{Nouv. ann. math.} (3)  {\bf  2} (1883), 535--542.

\bibitem{Ro} Gerhard Rosenberger, \"{U}ber die diophantische {G}leichung
              {$ax^{2}+by^{2}+cz^{2}=dxyz$},
\emph{J. Reine Angew. Math.}
{\bf 305} (1979), 122--125.
 
\bibitem{OEIS} N. J. A. Sloane,  \emph{The On-line Encyclopedia of Integer Sequences}, \url{https://oeis.org}.

 \bibitem{Sm} Leander W. Smith, Conditions governing numerical equality of perimeter, area, and volume, \emph{Math. Teacher} {\bf  58} (1965), no.~4,  303--307.

\bibitem{art} Victor Wang,  \emph{The Art of problem Solving}, \url{https://artofproblemsolving.com/community/c1461h1035155}.

\bibitem{Wie} Joseph Wiener, Henjin Chin and Hushang Poorkarimi, Involutions and problems involving perimeter and area, \emph{College Math. J.} {\bf 19} (1988), no.~3, 250--252.  

\bibitem{Yiu} Paul Yiu, Heronian triangles are lattice triangles,
\emph{Amer. Math. Monthly} {\bf 108} (2001),
no. 3, 261--263.

\bibitem{Yuan} Qiaochu Yuan,  \emph{Annoying Precision}, \url{https://qchu.wordpress.com/2009/07/02/square-roots-have-no-unexpected-linear-relationships/}.

\end{thebibliography}
{}

\end{document}